\documentclass[a4paper,11pt]{amsart}
\usepackage[english]{babel}

\usepackage{amsmath,amscd,amssymb,amsthm,amsxtra,enumerate}
\usepackage[foot]{amsaddr}
\usepackage{mathrsfs}
\usepackage{mathtools}
\usepackage{tikz}
\usepackage{color,graphicx}
\usepackage{epsfig,epstopdf}
\usepackage[margin=1.25in]{geometry}
\usepackage{enumerate}
\usepackage{hyperref}
\usepackage{subcaption}
\usepackage{wrapfig}
\usepackage{subcaption}
\usetikzlibrary{patterns}
\usetikzlibrary{decorations.pathreplacing}

\newcommand{\norm}[1]{\left\Vert#1\right\Vert}

\newcommand{\R}{\mathbb{R}}
\newcommand{\re}{\mathbb{R}}

\newcommand{\N}{\mathbb{N}}

\renewcommand{\H}{\mathcal{H}}

\newcommand{\xh}{\widehat{x}}

\newcommand\reallywidehat[1]{\arraycolsep=0pt\relax%
\begin{array}{c}
\stretchto{
  \scaleto{
    \scalerel*[\widthof{\ensuremath{#1}}]{\kern-.5pt\bigwedge\kern-.5pt}
    {\rule[-\textheight/2]{1ex}{\textheight}} 
  }{\textheight} %
}{0.5ex}\\           
#1\\                 
\rule{-1ex}{0ex}
\end{array}
}

\newcommand{\qqquad}{\qquad\qquad}
\newcommand{\qqqquad}{\qqquad\qqquad}

\newcommand{\mcl}{\mathcal{L}}

\newcommand{\ux}{\underline{x}}
\newcommand{\ox}{\overline{x}}
\newcommand{\pinv}{\pi^{\text{inv}}}

\newcommand{\ROmCl}{R^+_\lambda(\overline{\Omega})}

\newcommand{\ROmClk}{R^+_{\lambda_k}(\overline{\Omega})}
\newcommand{\ROmClkl}{R^+_{\lambda_{k_l}}(\overline{\Omega})}

\newcommand{\tx}{\Tilde{x}}
\newcommand{\ROmlO}{R^+_0(\overline{\Omega})}

\newcommand{\RPmL}{R^+_{\lambda}(\partial \Omega)}

\newcommand{\HP}{\text{HP}}

\DeclareMathOperator{\diam}{diam}

\newtheorem{thm}{Theorem}[section]

\newtheorem{lem}[thm]{Lemma}

\theoremstyle{definition}
\newtheorem{defn}[thm]{Definition}
\newtheorem{rem}[thm]{Remark}

\newtheorem{assum}[thm]{Assumption}

\numberwithin{equation}{section}

\allowdisplaybreaks

\author[A. Biswas]{Animesh Biswas}

\address{Department of Mathematics \\
{Missouri State University} \\
901 S. National Ave, Springfield\\
MO 65897, USA}
\email{ab7e@missouristate.edu}

\author[M. D. Foss]{Mikil D. Foss}
\address{Department of Mathematics \\
{University of Nebraska-Lincoln} \\
203 Avery Hall, Lincoln\\
NE 68588, USA}
\email{mikil.foss@unl.edu}

\author[P. Radu]{Petronela Radu}
\address{Department of Mathematics \\
{University of Nebraska-Lincoln} \\
203 Avery Hall, Lincoln\\
NE 68588, USA}

\email{pradu@unl.edu}

\thanks{M.D.F.'s and P.R.'s work was supported by the awards NSF-DMS 1716790 and NSF-DMS 2109149.}

\begin{document}

\title[Nonlocal Curvature]{Nonlocal Ordered Mean Curvature with Integrable Kernel}
\maketitle


\begin{abstract}
In this paper we introduce and study the concept of nonlocal ordered curvature. In the classical (differential) setting, the problem was introduced by Nirenberg and Li, where they conjectured that if a bounded, smooth surface has its mean curvature ordered in a particular direction, then the surface must be symmetric with respect to some hyperplane orthogonal to that direction. The conjecture was proved by Li et al in 2022. Here we study the counterpart problem in the nonlocal setting, where the nonlocal mean curvature of a set $\Omega$, at any point $x$ on its boundary, is defined as $H_\Omega^J(x) = \int_{\Omega^c} J(x-y) dy - \int_\Omega J(x-y) dy$ and the kernel function $J$ is radially symmetric, non-increasing, integrable and compactly supported. Using a generalization of Alexandrov's moving plane method, we prove a similar result in the nonlocal setting.  
\end{abstract}


\providecommand{\subjclass}[1]
{
  \small	
  {\textit{MSC2010:}} #1
}


\section{Introduction}

Curvature is a fundamental concept in physics and science, as it plays a crucial role in various areas such as classical mechanics, general relativity, optics, image processing, fluid dynamics. In particular, the curvature of surfaces can affect the mechanical, electrical, and optical properties of materials, so curvature effects need to be taken into account when designing and predicting behavior of new materials. The recently introduced concept of nonlocal curvature provides a framework for measuring the bending of a surface under little or no smoothness assumptions, while connecting to classical curvature as the horizon of interaction converges to zero.

 In this paper we consider {\it the ordered nonlocal mean curvature problem}, which is the study of surfaces for which the nonlocal mean curvature is monotonically non-decreasing in a particular direction. For a compact, connected, hypersurface $M$ embedded in $\R^n$, we define the classical (or local) mean curvature by $H(x) = \displaystyle\frac{1}{n-1} \sum_{i=1}^{n-1} k_i (x)$, where $k_i(x)$ are the principle curvatures of $M$ at the point $x$ in the $x_i$ direction.  

The concept of ordered curvature, in the classical/differential framework, was motivated by the constant mean curvature problem studied by Alexandrov \cite{Alexandrov}. In \cite{Li-1} Li first proved that if the mean curvature $H:M \to \R$, has a $C^1$-extension $F:\R^n \to \R$ such that $F$ has a non-positive partial derivative in a particular direction, say $x_n$ (meaning that $H$ is non-increasing in the $x_n$ component), then $M$ is symmetric with respect to some hyperplane $x_n= c$. In \cite{Li-2, Li-3} Li and Nirenberg studied the similar problem, under a weaker  assumption, namely
 \begin{assum} \label{assumption1}
     We write $x=(\xh, x_n)\in\R^n$. Let $\Omega$ be an open bounded set in $\R^n$, bounded by the hypersurface $M$, i.e $M = \partial \Omega$.  Then for any two points $x=(\xh, x_n), y=(\widehat{y}, y_n)\in M$, with $\xh = \widehat{y}$ and $x_n<y_n$, if $\{\xh, \theta x_n + (1-\theta) y_n \} \in \Omega$ for $0<\theta<1$, then we find $H(x) \leq H(y)$.
 \end{assum}
 Note that this ordered mean curvature assumption requires the surface to be at least twice-differentiable.  In \cite{Li-2, Li-3}, the authors showed that this condition is not sufficient to prove the set to be symmetric by providing a counterexample. However, they proved that the result holds true for one-dimensional surfaces, provided $M$ stays to one side of some tangent hyperplane parallel to $x_n$ \cite{Li-2}. In higher dimensions, additional conditions regarding a local convexity of the surface and the size of its tangential intersection with a $x_n=c$ plane were needed to obtain the symmetry of the set \cite{Li-3} with respect to a hyperplane $x_n = c$. In \cite{Li-4} Li together with Yan and Yao further generalized the results proving that a compact $C^2$ hypersurface is symmetric with respect to a hyperplane $x_n =c$ if it satisfies the ordered curvature condition and the following assumption:
 \begin{assum}
     There exists some constant $\delta>0$ such  that for every $x=(\xh, x_n) \in \partial \Omega$ with a horizontal unit outer normal $N=\langle \widehat{N}, 0 \rangle$, the vertical cylinder $|\xh - (\xh + \delta {N})| = \delta$ has an empty intersection with $\Omega$. 
 \end{assum}

 The above results have been obtained in the differential curvature setting. With the introduction and growing interest in nonlocal theories, counterparts to different geometrical concepts were introduced, in particular, the nonlocal perimeter and curvature. A nonlocal version of mean curvature was introduced by Caffarelli, Roquejofre, and Savin in 2009 in \cite{Caffa-Roq-Savin}. In this seminal paper, the authors defined the nonlocal curvature associated with a set $\Omega$, at {\it any} point $x \in \R^n$ as 
\begin{align}\label{eq:nlc_caffa}
    H_\Omega^s(x): = \int_{\re^n}\frac{\chi_{\Omega^c}(y)-\chi_{\Omega}(y)}{|x-y|^{n+2s}} \, dy,
\end{align}
where $\chi_A$ denotes the characteristic function associated with a set $A$. First, note that the above integral is well-defined by considering the principal value for sets $\Omega$ with sufficiently smooth boundaries. Second, we see that the concept can be defined at points $x$ which are not necessarily on the boundary. Following the introduction of this new geometric measure, the paper \cite{Caffa-Roq-Savin} also motivated several research problems including the constant mean curvature and the nonlocal minimal surface problems. Shortly after, the constant mean curvature problem in this nonlocal setting was solved in \cite{Cabre-Fall-Sola-Weth, Figalli}, showing that a (sufficiently smooth) surface of constant nonlocal mean curvature must be a ball. 

In the same spirit and motivated by similar geometric aspects, as well as by a desire to remove regularity assumptions for the surface considered, a different variation of the nonlocal curvature was considered, where the (highly) singular kernel $\displaystyle\frac{1}{|x|^{n+2s}}$ was replaced by an integrable non-negative kernel $J(x)$ supported inside a ball of radius $r$ (which could be infinite) - see \cite{curvaturepaper}. Very recently, Bucur and Fragala studied the constant curvature problem for such kernels in \cite{Bucur1, Bucur2}, and they found the solution to be a union of disjoint balls situated at a distance larger than $r$. 

Whether the kernel is non-integrable or integrable drives very different analytical arguments. In the case of highly singular kernels, one needs a minimum regularity assumption to be imposed on the surface; usually, $\partial \Omega$ must be $C^{1,\alpha}$ with $\alpha$ dependent on $s$, the order of the fractional curvature. On the other hand, for integrable kernels, it is sufficient for the set $\Omega$ to be measurable. In fact, the solution of the constant curvature problem in \cite{Bucur1, Bucur2} assumes the set to only be measurable and satisfy a certain nondeneracy condition. More precisely, they defined a set to be nondegenerate if 
\begin{align}\label{eq:non_dege_J}
    \inf_{x_1, x_2 \in \partial^* \Omega} \frac{\int_{\Omega} |J(x_1-y) - J(x_2-y)| \, dy }{\norm{x_1 - x_2}} > 0,
\end{align}
where $\partial^* \Omega$ is the essential boundary of the set $\Omega$.

With an eye towards considering general domains, we are interested here in studying the ordered nonlocal curvature in the integrable kernel setting. For simplicity, we consider kernels which are also bounded on the support set. A prototypical example of this kernel would be
\begin{equation}\label{eq:exam_kernel}
   z\mapsto {\chi_{B_{r}}(z)},
\end{equation}
for some $r>0$. Here $B_{r}\subseteq\re^n$ is the open ball with radius $r$ centered at the origin. More generally, we make the following
\begin{assum}\label{assum_monotonicty} The integrable kernel $J\in L^1(\re^n)$ satisfies the (J1 - J3) below and either (J4a) or (J4b):
\begin{itemize}
    \item[(J1)] \textbf{rotational symmetry:} there exists $\mu:[0,\infty)\to [0, \infty)$ such that $J(z)=\mu(|z|)$ for all $z\neq0$;
    \item[(J2)] \textbf{compact support:} there exists $0<r<\infty$ such that $\mu(\rho)=0$ for all $\rho\ge r$;
    \item[(J3)] \textbf{radially non-increasing:} if $0<\rho_2<\rho_1\le r$, then $\mu(\rho_1)\leq \mu(\rho_2)$.
\end{itemize}
Additionally, we require that the kernel $J$ be bounded. More specifically, we further assume that it satisfies one of the following cases (each of them would yield boundedness):
\begin{enumerate}
    \item[(J4a)] $J(z) = \chi_{B_r}(z)$
    \item[(J4b)] $\mu(\rho)$ is continuously differentiable on $[0,r],\, \mu(r)=0,$ and $\mu'(\rho)<0$ for all $0<\rho<r$.
\end{enumerate}
\end{assum}
Given a measurable set $\Omega \subseteq \re^n$, we define its nonlocal perimeter (with kernel $J$; see \cite{curvaturepaper}) as
\begin{equation}\label{eq:non_peri}
 P^J(\Omega): = \int_{\Omega} \int_{\Omega^c} J(x-y) dx\, dy,   
\end{equation}
and the nonlocal curvature at $x \in \re^n$ as
\begin{equation}\label{curvat}
    H^J_{\Omega}(x)
    :=\int_{\re ^n} J(x-y)(\chi_{\Omega^c}(y) - \chi_{\Omega}(y)) \, dy
    =\int_{\re ^n} J(x-y)\tau_\Omega(y)dy,
\end{equation}
where $\tau_\Omega = \chi_{\Omega^c}- \chi_{\Omega}$. With a kernel given by $J(x) = \chi_{B_r}(x)$, we see that the nonlocal curvature is simply given by $H^J_\Omega (x) = |B_r(x) \cap \Omega^c| - |B_r(x) \cap \Omega|$. In other words, the nonlocal curvature at a point $x$ is measured by the difference in volumes of the ball outside the set $\Omega$ and respectively, inside the set $\Omega$ (see Figure \ref{fig:nonlocal_curv}).
\begin{figure}
\centering
\begin{tikzpicture}[scale=6.05]
\draw[-latex,black] (-.65,-.3)--(-.65,.45) node[above]{$x_n$};
\draw[fill=blue,opacity=.2] (-.3,.3)--(.1,.3)--(0,0)--(.45,.15)--(.2,.42)--(.15,-.14)--(-.3,-.3);
\draw[blue] (-.3,.3)--(.1,.3)--(0,0)--(.45,.15)--(.2,.42)--(.15,-.14)--(-.3,-.3)--(-.3,.3);
\draw [red, fill=red!50,opacity=.3] (0,0) circle (0.34 cm);
\begin{scope}
    \clip (-.3,.3)--(.1,.3)--(0,0)--(.45,.15)--(.2,.42)--(.15,-.14)--(-.3,-.3)--(-.3,.3);
    \draw [red, fill=blue!40!green,opacity=.5] (0,0) circle (0.34 cm);
\end{scope}
\node at (.34,.2){$\Omega$};
\draw[fill=black] (0,0) circle(.007) node[below] {$x$};
\draw[black,thick] (-.3,-.3)  to node[left,xshift=-3.9] {\tiny\parbox{46pt}{lateral\\[-2pt] boundary\\[-2pt]points (P3)}} (-.3,.3);
\draw[thick,black,decorate,decoration=brace,xshift=-1.8] (-.3,-.3)--(-.3,.3);
\draw[fill=black] (.325,.285) circle (.007);
\node[black,above right] at (.325,.285) {\tiny\parbox{50pt}{top boundary \\[-2pt]point (P1)}};
\draw[fill=black] (.45,.15) circle (.007);
\node[black,right] at (.45,.15) {\tiny\parbox{65pt}{isolated boundary\\[-2pt]point (P4)} };
\draw[fill=black] (.225,.075) circle (.007);
\draw[stealth-] (.24,.065)--(.4,0) node[black,right]{\tiny\parbox{65pt}{bottom boundary\\[-2pt]point (P2)} };
\end{tikzpicture}
\caption{Types of boundary points (P1 - P4). The nonlocal curvature, $H^J_\Omega(x)$, is the difference between the $J$-weighted ``outer" volume (red) and its $J$-weighted ``inner" volume (green).}
\label{fig:nonlocal_curv}
\end{figure}
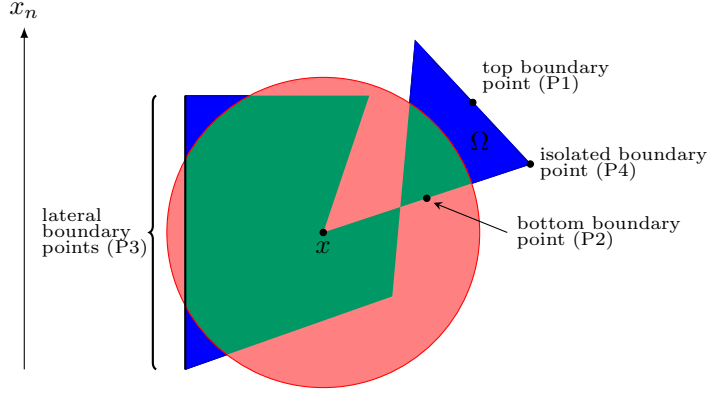

\begin{figure}
\centering
\begin{tikzpicture}[scale=.9]
    \draw[-latex,black,thick] (-.8,-.5)--(-.8,5) node[above]{$x_n$};
    \draw[thick,black] (-.5,-.9)--(0,-.9);
    \draw[blue,dashed] (0,-.8)--(0,0);
    \draw[blue,dashed] (7,-.8)--(7,0);
    \draw[thick,black,-latex] (7,-.9)--(8,-.9) node[right]{$\re^{n-1}$};
    \draw[fill=blue!70,opacity=.2]
    (0,0).. controls +(left:0cm) and +(left:0cm) .. (0,1)
    .. controls +(right:0cm) and +(right:0cm) .. (3,5)
    .. controls +(right:1cm) and +(up:1.5cm) .. (6,3)
    .. controls +(down:.25cm) and +(right:.5cm) .. (4,2)
    .. controls +(right:0cm) and +(right:0cm) .. (7,0)
    .. controls +(left:0cm) and  +(right:0cm) .. (0,0);
    \draw[blue,line width=.07,thick]
    (0,0).. controls +(left:0cm) and +(left:0cm) .. (0,1)
    .. controls +(right:0cm) and +(right:0cm) .. (3,5)
    .. controls +(right:1cm) and +(up:1.5cm) .. (6,3)
    .. controls +(down:.25cm) and +(right:.5cm) .. (4,2)
    .. controls +(right:0cm) and +(right:0cm) .. (7,0)
    .. controls +(left:0cm) and  +(right:0cm) .. (0,0);
    \draw[black,dashed,opacity=.5] (5,5.1)--(5,-1);
    \node[below,yshift=.3] at (5,-1) {$\widehat{x}$};
    \node at(1.5,.5){$\Omega$};
    
    \draw[red,fill=red!50,opacity=.3] (5,0)circle(.5cm);
    \draw[red,fill=red!50,opacity=.3] (5,1.33)circle(.5cm);
    \draw[red,fill=red!50,opacity=.3] (5,2.31)circle(.5cm);
    \draw[red,fill=red!50,opacity=.3] (5,4.5)circle(.5cm);
    \begin{scope}
        \clip  (0,0).. controls +(left:0cm) and +(left:0cm) .. (0,1)
    .. controls +(right:0cm) and +(right:0cm) .. (3,5)
    .. controls +(right:1cm) and +(up:1.5cm) .. (6,3)
    .. controls +(down:.25cm) and +(right:.5cm) .. (4,2)
    .. controls +(right:0cm) and +(right:0cm) .. (7,0)
    .. controls +(left:0cm) and  +(right:0cm) .. (0,0);
        \draw[blue,fill=blue!40!green,opacity=.5] (5,0)circle(.5cm);
        \draw[blue,fill=blue!40!green,opacity=.5] (5,1.33)circle(.5cm);
        \draw[blue,fill=blue!40!green,opacity=.5] (5,2.31)circle(.5cm);
        \draw[blue,fill=blue!40!green,opacity=.5] (5,4.5)circle(.5cm);
    \end{scope}
    \node[above left, style={scale=.9}] at(5.05,0){\tiny$A^1$};
    \node[below left, style={scale=.9}] at(5.05,1.4){\tiny$B^1$};
    \node[above left, style={scale=.9}] at(5.05,2.27){\tiny$A^2$};
    \node[below left, style={scale=.9}] at(5.05,4.6){\tiny$B^2$};
    \draw[dashed,thin,opacity=.5] (5,0)--(-1,0) node[left,opacity=1]{$a^1$};
    \draw[dashed,thin,opacity=.5] (5,1.33)--(-1,1.33) node[left,opacity=1]{$b^1$};
    \draw[dashed,thin,opacity=.5] (5,2.31)--(-1,2.33) node[left,opacity=1]{$a^2$};
    \draw[dashed,thin,opacity=.5] (5,4.5)--(-1,4.5) node[left,opacity=1]{$b^2$};
    \draw[fill=black] (5,0)circle(.04);
    \draw[fill=black] (5,1.33)circle(.04);
    \draw[fill=black] (5,2.31)circle(.04);
    \draw[fill=black] (5,4.5)circle(.04);
    \draw[thick,blue] (0,-.9)--(7,-.9) node[midway,below,yshift=-2.4,black]{$R=\pi(\overline{\Omega})$};    \draw[thick,blue,decorate,decoration=brace,yshift=-2.5] (7,-.9)--(0,-.9);
\end{tikzpicture}  
\caption{The nonlocal curvature is ordered in the $x_n$-direction: $[a^1,b^1],[a^2,b^2]\subseteq\pinv_{\overline{\Omega}}(\xh)$, so $H^J_\Omega(A^1)\le H^J_\Omega(B^1)$ and $H_\Omega^J(A^2)\le H_\Omega^J(B^2)$.}
\label{fig:ordered_curvature}
\end{figure}
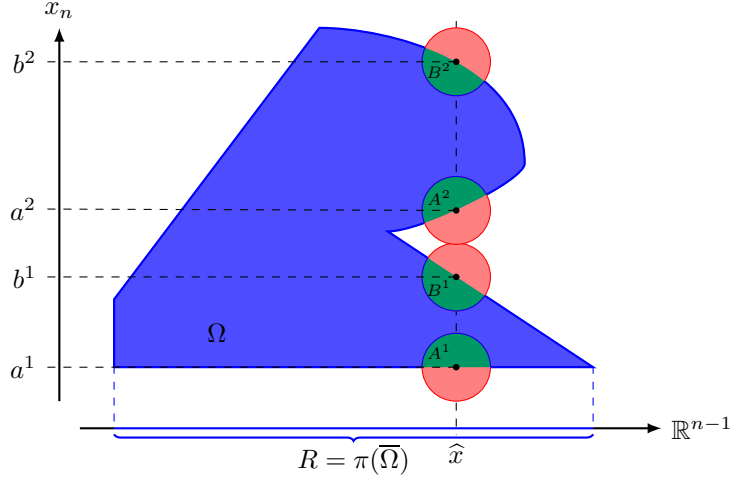
As these definitions only require the set $\Omega$ to be measurable, they are applicable to very general domains which could be quite singular. Thus the concepts of nonlocal perimeter and curvature are amenable to the study of rough images and objects, which can be found in a variety of applications, e.g. fracture and singular deformations (see \cite{suo2012new} for a method of using curvature to predict fracture). Nonlocal models in general, have been successfully employed in a variety of applications, and the mathematical theory behind them is seeing an explosive growth (\cite{B-C-F-R, BDS, B-S, blanccharro2022asymptotic, F-R-Y, H-B-A}). Similar to the classical (differential) theory, it has been shown that the nonlocal curvature is obtained by minimizing the nonlocal perimeter functional. More specifically, nonlocal minimal sets have zero nonlocal curvature (see Definition 3.3 and Theorem 3.4 in \cite{curvaturepaper}).

Before listing the assumptions on the set $\Omega$, we need to introduce some additional notation. Let $\pi:\R^n\ni(\xh,x_n)\mapsto\xh\in\R^{n-1}$ denote the projection map. Given a set $E\subseteq\R^n$ and $\xh\in\R^{n-1}$, define the set
\[
    \pinv_E(\xh):=\{x_n:(\xh,x_n)\in E\}.
\]
We will also define $R=\pi(\overline{\Omega})$. In this paper, we will work with sets which satisfy the following:
\begin{assum}\label{assum_domain}
Let $\Omega$ be an open bounded connected set such that:

\begin{enumerate}
    \item[($\Omega1)$] {\bf Monotone nonlocal $J-$curvature condition:} For any $x=(\xh,x_n),y=(\widehat{y},y_n) \in \partial \Omega$,  satisfying $\xh=\widehat{y}$, $x_n < y_n$ and $\theta x + (1 - \theta)y \in \overline{\Omega}$ for all $0 \leq \theta \leq 1$,  we assume
    $$H^J_\Omega(x) \leq H^J_\Omega(y).$$
    See Figure~\ref{fig:ordered_curvature} for an illustration of this condition at a fixed projection point $\widehat{x}$. Note that the condition must be satisfied by pairs of endpoints of line segments within $\overline{\Omega}$.
    \item [($\Omega2)$] {\bf Zero boundary measure:}  The boundary $\partial \Omega$ has zero $\mathcal{H}^n$-measure.
    \item [($\Omega3)$] {\bf Countable interval projections:} We assume that for each $\xh \in R$, the set $\pinv_{\overline{\Omega}}(\xh)\subseteq\re$ is a countable union of closed intervals. 
 
    \item [($\Omega4)$] {\bf Nondegeneracy of the boundary/large diameter:} Finally, in the case of the constant kernel $J(z) = \chi_{B_r}(z)$ we also need the following nondegeneracy condition: for any $x \in \partial \Omega$, we require $\mathcal{H}^{n-1}(\partial B_r(x) \cap \Omega) >0$. As $\Omega$ is connected, it is sufficient to assume that $\diam(\Omega)>2r$. 
\end{enumerate}    
\end{assum}
\begin{rem}\label{rem: example}
Open connected sets may fail to possess the property ($\Omega3$). For example, let $\mathcal{C}\subseteq[0,1]$ denote the Cantor ternary set. For each $y\in\mathcal{C}$, let $T_y\subseteq[0,1]\times[0,1]$ denote the open triangle with vertices at $(0,0)$, $(0,1)$, and $(1,y)$. Define the open set $\Omega=\bigcup_{y\in\mathcal{C}}T_y$. Then, $\pinv((1))=\mathcal{C}$ is a totally disconnected uncountable set. Therefore, it violates ($\Omega3$).



\end{rem}
Our main result is
\begin{thm}\label{thm:main}
    Let $\Omega$ be an open set that satisfies all the conditions of Assumptions~\ref{assum_monotonicty} and~\ref{assum_domain}. Then, there is a $\lambda_0\in\re$ such that $\Omega$ is symmetric across the hyperplane $\{(\xh,x_n)\in\R^n: x_n = \lambda_0\}$.
\end{thm}

  There are two main steps to prove Theorem \ref{thm:main}. In the first step, we prove the fact that if the curvature values at different points on the boundary are ordered along the $x_n$-direction, then they must be equal. This result is similar to its classical counterpart \cite[Proposition 3]{Li-3}, and is also obtained by employing a variational method. However, since our set has minimal regularity assumptions, proving this step becomes much more challenging than in the classical framework. The second step in our proof will employ the Alexandrov moving-plane method. In \cite{Bucur1, Bucur2}, the authors introduced a generalized version of the Alexandrov moving plane method to solve the constant mean curvature problem. Although the main idea in the ordered curvature setting is similar, the results and arguments of \cite{Bucur1, Bucur2} cannot be applied directly to the ordered curvature problem, as the curvature is not necessarily constant along the boundary. More precisely, we can only use the fact that the curvatures at different boundary points connected by a straight line along the $x_n$-axis and in $\overline{\Omega}$ are equal. To overcome this obstacle, we develop new arguments that employ the monotonicity properties of the curvature.

\section{{Ordered Curvature Implies Pairwise Equal Curvature}}
Throughout the paper we denote the points in $\R^n$ by $x=(\xh,x_n)\in \R^{n-1}\times \R$. Unless explicitly mentioned otherwise, we will simply write $H^J$ instead of $H^J_\Omega$. Given $\xh\in R$, there is a countable index set $\Pi_{\xh}\subseteq\N$ such that
\[
    \pinv_{\overline{\Omega}}(\xh) = \bigcup_{i \in {\Pi}_{\xh}} [\ux^i, \overline{x}^i].
\]
Without loss of generality, we may assume that $\ux^i>\overline{x}^{i+1}$, for each $i\in\Pi_{\xh}$. Thus, $\{[\ux^i,\overline{x}^i]\}_{i\in\Pi_{\xh}}$ consists of disjoint intervals. Note that we may have $\ux_i=\overline{x}^i$, for some $i\in\Pi_{\xh}$.
We observe that a point $x=(\xh, x_n) \in \partial \Omega$ may be classified as follows:
\begin{enumerate}[(P1)]
    \item There exists $\delta_0 =\delta_0(x)>0$ such that the line segment $\{\xh\}\times[x_n- \delta,x_n] \subseteq \overline{\Omega}$, while the line segment $\{\xh\}\times(x_n,x_n+\delta) \subseteq \overline{\Omega}^c$ for every $\delta <\delta_0$. We call this point a {\it top boundary point}. The collection of top boundary points of $\Omega$ is denoted by $M_1$.
    \item There exists $\delta_0 =\delta_0(x)>0$ such that the line segment $\{\xh\}\times(x_n- \delta, x_n) \subseteq \overline{\Omega}^c$, while $\{\xh\}\times[x_n, x_n+\delta] \subseteq \overline{\Omega}$ for all $\delta<\delta_0$. Such a  point will be called a {\it bottom boundary point}. The collection of bottom boundary points of $\Omega$ is denoted by $M_2$.
    \item There exists $\delta_0 =\delta_0(x)>0$ such that $\{\xh\}\times[x_n- \delta, x_n +\delta] \subseteq \overline{\Omega}$ for all $\delta<\delta_0$. This type of point will be referred to as a {\it lateral boundary point}. The collection of lateral boundary points of $\Omega$ is denoted by $M_3$.
    \item There exists $\delta_0 =\delta_0(x)>0$ such that $(\{\xh\}\times(x_n- \delta, x_n+\delta))\setminus\{x\}\subseteq \overline{\Omega}^c$ for all $\delta<\delta_0$. We call this {\it a (directionally) isolated point}. The set of all isolated points is denoted by $M_4$.
\end{enumerate}
Note that the above classification is dependent on the choice of the direction $x_n$. We also observe that for the example in Remark~\ref{rem: example}, each $(1,x_n)\in\partial\Omega$ is an isolated point.

Given the above nomenclature and the fact that $\pinv_{\overline{\Omega}}(\xh) = \bigcup_{i \in {\Pi}_{\xh}} [\ux^i, \overline{x}^i]$, whenever $\ux^i \neq \overline{x}^i$ for some $i$, we see that $\ux^i, \overline{x}^i$ are bottom and top boundary points, respectively, connected via a line segment (parallel to $x_n$-axis) in $\overline{\Omega}$. In such a situation, we write $\ux^i \sim \overline{x}^i$. On the other hand, if $\ux^i = \overline{x}^i$ for some $i$, then we have an isolated point.

\begin{thm}\label{thm:monotone_implies_equality}
    Let $\Omega$ be a bounded, open and connected set that satisfies Assumption \ref{assum_domain}.  Then, for any two boundary points $\ux$ and $\overline{x}$ such that $\ux \sim \overline{x}$, we have $H^J(\ux) = H^J(\overline{x})$. 
\end{thm}

To demonstrate the proof strategy, we first establish this theorem in the special case where $\pinv_{\overline{\Omega}}(\xh)=[\ux,\overline{x}]$ is a single interval for every $\xh \in R$.
\begin{thm}\label{lem: simple}
     Let $\Omega$ be a bounded, open and connected set that satisfies the monotonicity condition for the nonlocal mean curvature, as mentioned in \ref{assum_domain}. Additionally, assume that $\pinv_{\overline{\Omega}}(\xh)$ is a closed interval for every $\xh \in R$.  
      Then for any two boundary points $\ux$ and $\overline{x}$ such that $\ux \sim \overline{x}$, we have $H^J_\Omega(\ux) = H^J_\Omega(\overline{x})$.
\end{thm}
\begin{proof}[{Proof}]
The case $\ux= \overline{x}$ is trivial. Assuming otherwise, we translate $\Omega$ in the $x_n$-direction by an amount $t$. Denote the translation of $\overline{\Omega}$ by $\overline{\Omega}_t$, and the components of the relative complements by $E_t: = \overline{\Omega}_t \setminus \overline{\Omega}$ and $F_t: = \overline{\Omega} \setminus \overline{\Omega}_t$ (see Figure \ref{fig:tran}).
\begin{figure}{\label{fig:tran}}
\begin{center}
\begin{tikzpicture}[scale=.9]
    \draw[-latex,black,thick] (-.5,-.5)--(-.5,6.1) node[above]{$x_n$};
    \draw[thick,black,-latex] (-.5,-.8)--(8,-.8) node[right]{$\re^{n-1}$};
    \draw[red!70!black,line width=.7pt,fill=red!40,opacity=.8,thick]
    (0,0).. controls +(left:0cm) and +(left:0cm) .. (0,1)
    .. controls +(right:0cm) and +(right:0cm) .. (3,5)
    .. controls +(right:1cm) and +(up:1.5cm) .. (6,3)
    .. controls +(down:.25cm) and +(right:.5cm) .. (4,2)
    .. controls +(right:0cm) and +(right:0cm) .. (7,0)
    .. controls +(left:0cm) and  +(right:0cm) .. (0,0);
    \node[red!70!black] at(1.5,.5){$\overline{\Omega}$};
    \begin{scope}[yshift=25pt]
        \draw[blue!80!black,line width=.7pt,fill=blue!60,opacity=.4,thick]
    (0,0).. controls +(left:0cm) and +(left:0cm) .. (0,1)
    .. controls +(right:0cm) and +(right:0cm) .. (3,5)
    .. controls +(right:1cm) and +(up:1.5cm) .. (6,3)
    .. controls +(down:.25cm) and +(right:.5cm) .. (4,2)
    .. controls +(right:0cm) and +(right:0cm) .. (7,0)
    .. controls +(left:0cm) and  +(right:0cm) .. (0,0);
    \end{scope}

    \begin{scope}
        \clip (0,0).. controls +(left:0cm) and +(left:0cm) .. (0,1)
    .. controls +(right:0cm) and +(right:0cm) .. (3,5)
    .. controls +(right:1cm) and +(up:1.5cm) .. (6,3)
    .. controls +(down:.25cm) and +(right:.5cm) .. (4,2)
    .. controls +(right:0cm) and +(right:0cm) .. (7,0)
    .. controls +(left:0cm) and  +(right:0cm) .. (0,0);

    \begin{scope}[yshift=25pt]
        \draw[red!70!black,line width=.7pt,fill=black!20,opacity=.4,thick]
    (0,0).. controls +(left:0cm) and +(left:0cm) .. (0,1)
    .. controls +(right:0cm) and +(right:0cm) .. (3,5)
    .. controls +(right:1cm) and +(up:1.5cm) .. (6,3)
    .. controls +(down:.25cm) and +(right:.5cm) .. (4,2)
    .. controls +(right:0cm) and +(right:0cm) .. (7,0)
    .. controls +(left:0cm) and  +(right:0cm) .. (0,0);
    \node[blue!70!black,opacity=1] at(1.5,.5){$\overline{\Omega}_t$};
    \end{scope}
    \end{scope}
    \draw[red!70!black,thick,-latex] (4,-.3)--(3.5,.5);
    \node[red!70!black,below right] at (3.9,-.1){$F_t=\overline{\Omega}\setminus\overline{\Omega}_t$};
    \draw[blue!70!black,thick,-latex] (5.5,5.6)--(4.4,5.3);
    \node[blue!70!black,right] at (5.5,5.7){$E_t=\overline{\Omega}_t\setminus\overline{\Omega}$};
\end{tikzpicture}  
\end{center}
\caption{Translation domains and relative complement components} {\label{fig:tran}}
\end{figure}
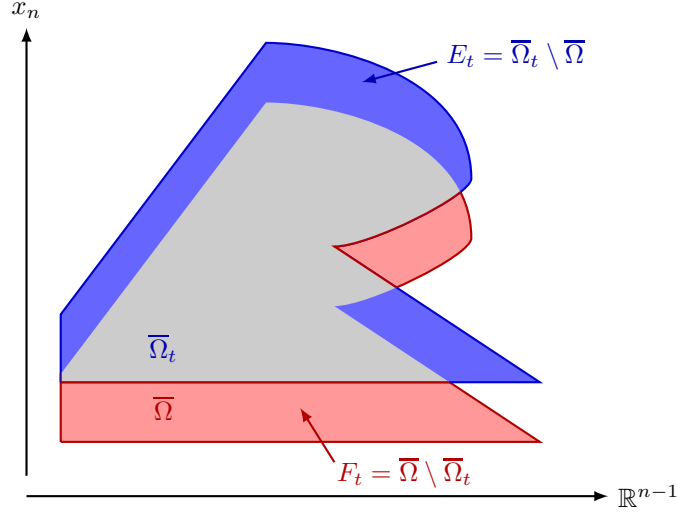

Then, using the definition of nonlocal perimeter, we have 
\begin{align}
\nonumber
   0&= P^J(\overline{\Omega}_t) - P^J(\overline{\Omega}) \\
\nonumber
   &= \int_{\overline{\Omega}} \int_{F_t} J(x-y) \,dy \,dx - \int_{F_t} \int_{\overline{\Omega}^c} J(x-y) \,dy \,dx -\int_{F_t} \int_{F_t} J(x-y) \,dy \,dx \\
\nonumber
   &\qquad - \int_{\overline{\Omega}} \int_{E_t} J(x-y) \,dy \,dx + 2 \int_{E_t}\int_{F_t} J(x-y) \,dy \,dx + \int_{E_t} \int_{\overline{\Omega}^c} J(x-y) \,dy \,dx \\
\nonumber
   &\qquad -\int_{E_t} \int_{E_t} J(x-y) \,dy \,dx \\
\label{eq: perim0}
   &= \int_{F_t} H^J(x) dx - \int_{E_t} H^J(x) dx + 2 \int_{E_t}\int_{F_t} J(x-y) \,dy \,dx \\
\nonumber
   &\qquad -\int_{F_t} \int_{F_t} J(x-y) \,dy \,dx
   - \int_{E_t} \int_{E_t} J(x-y) \,dy \,dx.
\end{align}
In the previous line we used the fact that {$\partial \Omega$ has measure $0$ by ($\Omega 2$)} and hence:
\begin{align*} \int_{\R^n}\big(\chi_{\overline{\Omega}^c}(y) - &\chi_{\overline{\Omega}}(y) \big) J(x-y) dy\\ &=\int_{\R^n}\big(\chi_{{\Omega}^c}(y) - \chi_{{\Omega}}(y) \big) J(x-y) dy - 2 \int_{\partial \Omega}J(x-y) dy = H^J(x).
\end{align*}
Next, we define
$$f_1(\xh): = \max \{x_n: (\xh, x_n) \in \partial \Omega \}
\quad\text{ and }\quad f_2(\xh): =  \min \{x_n: (\xh, x_n) \in \partial \Omega \}.$$
Recall that $R = \pi(\overline{\Omega})$. Let us define
\[
    R'_t: = \{ \xh \in R: f_1(\xh) - f_2(\xh) >t \}
    \quad\text{ and }\quad
    R''_t: = \{ \xh \in R: f_1(\xh) - f_2(\xh) \leq t \}.
\]
Hence,
$$\int_{E_t} H^J(x) dx = \int_{R'_t} \int^{f_1(\xh) +t}_{f_1(\xh)} H^J(\xh,x_n) dx_n \, d\xh + \int_{R''_t} \int^{f_1(\xh) +t}_{f_2(\xh) +t} H^J(\xh,x_n) dx_n \, d\xh $$
and
$$\int_{F_t} H^J(x) dx = \int_{R'_t} \int^{f_2(\xh) +t}_{f_2(\xh)} H^J(\xh,x_n) dx_n \, d\xh + \int_{R''_t} \int^{f_1(\xh) }_{f_2(\xh) } H^J(\xh,x_n) dx_n \, d\xh. $$
Using a change of variables we obtain
$$\int_{E_t} H^J(x) dx = \int_{R'_t} \int^{t}_{0} H^J(\xh,x_n+f_1(\xh)) dx_n \, d\xh + \int_{R''_t} \int^{f_1(\xh)}_{f_2(\xh)} H^J(\xh,x_n+t) dx_n \, d\xh, $$
and
$$\int_{F_t} H^J(x) dx = \int_{R'_t} \int^{t}_{0} H^J(\xh,x_n+f_2(\xh)) dx_n \, d\xh + \int_{R''_t} \int^{f_1(\xh) }_{f_2(\xh)} H^J(\xh,x_n) dx_n \, d\xh. $$
Thus, we get
\begin{align*}
    &\int_{E_t} H^J(x) dx - \int_{F_t} H^J(x) dx \\
    &\:=\int_{R'_t} \int^{t}_{0} H^J(\xh,x_n+f_1(\xh)) dx_n \, d\xh + \int_{R''_t} \int^{f_1(\xh)}_{f_2(\xh)} H^J(\xh,x_n+t) dx_n \, d\xh \\
    &\qqqquad - \int_{R'_t} \int^{t}_{0} H^J(\xh,x_n+f_2(\xh)) dx_n \, d\xh - \int_{R''_t} \int^{f_1(\xh) }_{f_2(\xh)} H^J(\xh,x_n) dx_n \, d\xh \\
    &\:= \int_{R'_t} \int^t_0 \big[ H^J(\xh, f_1(\xh)) - H^J(\xh, f_2(\xh)) \big] dx_n \,d\xh\\
    &\qqqquad + \int_{R'_t} \int^{t}_{0} \big[H^J(\xh,x_n+f_1(\xh))- H^J(\xh, f_1(\xh))\big] dx_n \, d\xh \\
    &\qqqquad - \int_{R'_t} \int^{t}_{0} \big[H^J(\xh,x_n+f_2(\xh))- H^J(\xh, f_2(\xh))\big] dx_n \, d\xh\\
    &\qqqquad+ \int_{R''_t} \int^{f_1(\xh)}_{f_2(\xh)} \big[H^J(\xh,x_n+t)-H^J(\xh, x_n) \big] dx_n \, d\xh.
\end{align*}
 Next we see that:
\begin{align*}
    H^J(x)-H^J(\overline{x})&= \int_{\Omega^c} (J(x-y)-J(\overline{x} -y)) \, dy - \int_\Omega (J(x-y)-J(\overline{x} -y)) \, dy \\
    &= -2\int_\Omega (J(x-y)-J(\overline{x} -y)) \, dy.
\end{align*}
Then, under the assumption (J4a), we have
\begin{align*}
    H^J(x)-H^J(\overline{x})   &=- 2\int_{\Omega \cap \Delta} (J(x-y)-J(\overline{x} -y)) \, dy,
\end{align*}
where $\Delta = (B_r(x) \setminus B_r(\overline{x})) \cup (B_r(\overline{x}) \setminus B_r(x))$. For $\delta>0$ sufficiently small, $|\overline{x}-x|<\delta$ implies $|\Delta| < w_n r^{n-1} \delta$. Thus, we obtain:
$$|H^J(x) - H^J(\overline{x})| \leq 4  \min \{ |\Delta|, |\Omega| \} \leq 4 w_n r^{n-1} \delta.$$ 
On the other hand, under the assumption (J4b), we have
\begin{align*}
    H^J(x)-H^J(\overline{x})
    &= -2\int_\Omega (J(x-y)-J(\overline{x} -y)) \, dy \\
    &= -2 \int_\Omega \nabla J(\xi-y) \cdot (x-\overline{x}) dy
\end{align*}
where $\xi = \tau x + (1-\tau)\overline{x}$ for some $0\leq \tau \leq 1$. Since $J$ is continuously differentiable, we get
\begin{equation}\label{eq:HLip}
    |H^J(x) - H^J(\overline{x})| \leq 2\sup\{ |\nabla J(z)|:z\in\R^n \}\cdot |x-\overline{x}| |\Omega|.
\end{equation}
Thus, both cases show that $H$ is uniformly continuous and 
$$|H^J(x) - H^J(\overline{x})| \leq C|x- \overline{x}|.$$
Using this bound, we find
$$|H^J(\xh,x_n + f_1(\xh)) - H^J(\xh, f_1(\xh))| \leq Ct $$
$$|H^J(\xh,x_n + f_2(\xh)) - H^J(\xh, f_2(\xh))| \leq Ct $$
and
$$|H^J(\xh,x_n+t)-H^J(\xh, x_n)| \leq Ct .$$
Recalling that for $\xh \in R''_t,$, we have $|f_1(\xh) - f_2(\xh)| \leq t$, we additionally obtain the bound
\begin{multline*}
    \bigg|\int_{R'_t} \int^{t}_{0} \left[H^J(\xh,x_n+f_1(\xh))- H^J(\xh, f_1(\xh))\right] dx_n \, d\xh \\
    - \int_{R'_t} \int^{t}_{0} \big[H^J(\xh,x_n+f_2(\xh))- H^J(\xh, f_2(\xh))\big] dx_n \, d\xh\\ 
     + \int_{R''_t} \int^{f_1(\xh)}_{f_2(\xh)} \big[H^J(\xh,x_n+t)-H^J(\xh, x_n) \big] dx_n \, d\xh\bigg| \leq 3C \mcl^{n-1} (R) t^2.
\end{multline*}
The monotone convergence theorem yields
$$ \int_{R'_t} \big[H^J(\xh, f_1(\xh)) - H^J(\xh, f_2(\xh)) \big] d \xh \to \int_{R} \big[H^J(\xh, f_1(\xh)) - H^J(\xh, f_2(\xh)) \big] d \xh\quad\text{ as }\quad t\to0. $$
Therefore, 
$$\lim_{t \to 0} \frac{1}{t} \bigg(\int_{E_t} H^J(x) dx - \int_{F_t} H^J(x) dx \bigg) = \int_{R} \big[H^J(\xh, f_1(\xh)) - H^J(\xh, f_2(\xh)) \big] d \xh. $$
Next, we consider the terms
\begin{multline}\label{eq: perimTerms}
    \left(2\int_{E_t}\int_{F_t}-\int_{F_t}\int_{F_t}-\int_{E_t}\int_{E_t}\right) J(x-y)dydx\\
    =\left(\int_{E_t}\int_{F_t}-\int_{E_t}\int_{E_t}\right)J(x-y)dydx
    +\left(\int_{E_t}\int_{F_t}-\int_{F_t}\int_{F_t}\right)J(x-y)dydx
\end{multline}
from~\eqref{eq: perim0}. Focusing on the first difference of integrals, define $I^J_t(x) = \int_{E_t} J(x-y) \, dy$ so
\[
    \left(\int_{E_t} \int_{E_t}-\int_{F_t}\int_{E_t}\right)
        J(x-y)dydx= \int_{E_t} I^J_t(x) dx - \int_{F_t} I^J_t(x) dx.
\]
It is evident that $I^J_t(x)$ is a uniformly continuous function using similar reasoning for $H^J(x)$. Repeating a similar analysis as before yields
$$\lim_{t \to 0} \frac{1}{t} \bigg(\int_{E_t} I^J_t(x) dx - \int_{F_t} I^J_t(x) dx \bigg) = \lim_{t \to 0}  \int_{R'_t} \big[I^J_t(\xh, f_1(\xh)) - I^J_t(\xh, f_2(\xh)) \big] d \xh. $$
We observe that  
$|E_t|, |F_t| \leq \mcl^{n-1}(R) t$ and
$$\int_{E_t} J(x-y) dy \leq \int_{B_\rho(x)} J(x-y) dy = \int_{B_\rho} J(y) dy,$$
where $\rho<r$ is chosen such that $|E_t| = |B_\rho|$, and we used the fact that $J$ is radially symmetric and non-increasing. Here $\rho$ depends only on $t$. As $J$ is integrable we have
$$
\int_{B_\rho} J(y) dy < \epsilon/2, \quad \text{ for }~ t, \text{ hence, } \rho \text{ small}, $$
and $\epsilon$ depends only on $\rho$, and hence on $t$.
This gives us 
$$|I^J_t(\xh, f_1(\xh)) - I^J_t(\xh, f_2(\xh)) | < \epsilon,$$
which finally implies
$$\lim_{t \to 0} \frac{1}{t} \bigg(\int_{E_t} I^J_t(x) dx - \int_{F_t} I^J_t(x) dx \bigg)=  \lim_{t \to 0}  \int_{R'_t} \big[I^J_t(\xh, f_1(\xh)) - I^J_t(\xh, f_2(\xh)) \big] d \xh =0 .$$
For the second difference of integrals in~\eqref{eq: perimTerms}, we can similarly prove
$$\lim_{t \to 0} \frac{1}{t}\left(\int_{F_t} \int_{F_t}- \int_{E_t} \int_{F_t}\right) J(x-y) \, dy \, dx= 0.$$
Collecting all the results, we conclude that
$$0 =  \lim_{t \to 0} \frac{1}{t} \bigg(P^J(\Omega_t) - P^J(\Omega) \bigg) = \int_{R} \big[H^J(\xh, f_1(\xh)) - H^J(\xh, f_2(\xh)) \big] d \xh.$$
By the ordered curvature assumption, $\xh \in R$, $H^J(\xh,f_1(\xh)) \geq H^J(\xh, f_2(\xh))$. This implies 
$$H^J(\xh, f_1(\xh)) = H^J(\xh, f_2(\xh)),$$ for all $\xh \in R$.  
\end{proof}

{Next we provide the proof of Theorem \ref{thm:monotone_implies_equality} which covers a more general case.} 
\begin{proof}{\bf Proof of Theorem \ref{thm:monotone_implies_equality}}
We start with the translation of $\Omega$ in the $x_n$-direction by an amount $t>0$. This gives us 
\begin{align}
\nonumber
    0 =&
    \int_{F_t} H^J(x) dx - \int_{E_t} H^J(x) dx + 2 \int_{E_t}\int_{F_t} J(x-y) dy dx\\
\nonumber
    &\qqqquad\qqqquad- \int_{F_t} \int_{F_t} J(x-y) dy dx- \int_{E_t} \int_{E_t} J(x-y) dy dx\\
\label{eq: initInt}
    =&\left(\int_{F_t}-\int_{E_t}\right) H^J(x)dx\\
\nonumber
    &\qqquad+\left(\int_{E_t}\int_{F_t}- \int_{F_t} \int_{F_t}\right)J(x-y)dydx
        +\left(\int_{E_t}\int_{F_t}- \int_{E_t} \int_{E_t}\right)J(x-y)dydx
\end{align}
We will argue that
\begin{equation}\label{eq: bndCurv}
    \lim_{t\to0^+}\frac{1}{t}\left|\left(\int_{F_t}-\int_{E_t}\right) H^J(x)dx\right|
    =\int_R\sum_{i\in\Pi_{\xh}}\left[H^J(\xh,\overline{x}^i)-H^{J}(\xh,\ux^i)\right]d\xh\ge0
\end{equation}
and that
\begin{equation}\label{eq: extraInt}
    \lim_{t\to0^+}\frac{1}{t}\left(\int_{E_t}\int_{F_t}- \int_{E_t} \int_{E_t}\right)J(x-y)dydx,
    \lim_{t\to0^+}\frac{1}{t}\left(\int_{E_t}\int_{F_t}- \int_{F_t} \int_{F_t}\right)J(x-y)dydx=0.
\end{equation}
It follows then, that for each $i\in\Pi_{\xh}$, we have  $H^J(\xh,\ox^i)=H^J(\xh,\ux_i)$, which is equivalent to the theorem's statement.

Now, let us define, for each $\xh \in \R^{n-1}$, $l_{\xh}(s)$ to be the line segment, parallel to the $x_n$-axis, that runs from $(\xh, -\infty)$ to $(\xh, s)$ for some $s \in \R$. Additionally, define
$$\lambda(\xh, s): = \mathcal{H}^1(l_{\xh}(s) \cap \overline{\Omega}). $$
Then $\lambda(\xh,s)$ is a monotone function of $s$ and hence $\lambda$ is differentiable with respect to $s$ for a.e. $s\in\re$, with
$$\lambda'(\xh,s)=
\begin{cases}
    1, \quad \text{if}~(\xh,s) \in \Omega \\
    0, \quad \text{if}~(\xh,s) \notin \overline{\Omega}. 
\end{cases}
$$
In fact, assumption ($\Omega3$) implies $\lambda'$ exists outside a countable set. Thus, we may identify $\lambda'(\xh,\cdot)=\chi_{\overline{\Omega}}(\xh,\cdot)$ and define
\[
    f_t(\xh, s):=\lambda'(\xh,s-t)-\lambda'(\xh,s)=
    \left\{\begin{array}{rl}
        1, & \text{if}~(\xh,s)
            \in E_t = \overline{\Omega}_t \setminus \overline{\Omega} \\
        -1, & \text{if}~(\xh,s) \in F_t = \overline{\Omega} \setminus \overline{\Omega}_t\\
        0, & \text{if}~(\xh,s)\in \overline{\Omega}\cap\overline{\Omega}_t\\
        0, & \text{if}~(\xh,s)\in\overline{\Omega}^c\cap\overline{\Omega}_t^c.
    \end{array}\right.
\]
With this, we can write
\begin{align*}
    \int_{E_t} H^J(x) dx - \int_{F_t} H^J(x) dx &= \int_{R} \int^\infty_{-\infty} H^J(\xh,s) \big( \lambda'(\xh,s-t)-\lambda'(\xh,s) \big) \, ds \, d\xh \\
    &= \int_{R} \int^\infty_{-\infty} \big(H^J(\xh, s+t) - H^J(\xh,s)\big) \lambda'(\xh,s) \, ds \, d \xh. 
\end{align*}
Let us focus on a fixed $\xh \in R$. Recall $\pinv(\xh) = \bigcup_{i \in \Pi_{\xh}} [\ux^i, \overline{x}^i]$. With $t>0$, for each $i\in\Pi_{\xh}$, set
$$A_i(t): = \pinv_{E_t}(\xh) \cap [\ux^i +t, \overline{x}^i +t ], \quad \text{and}\quad B_i(t): = \pinv_{F_t}(\xh) \cap [\ux^i, \overline{x}^i].$$ 
Put $\alpha_i(t):=\mathcal{H}^1(A_i(t))$ and $\beta_i(t):=\mathcal{H}^1(B_i(t))$. Then we can easily see that $0 \leq \alpha_i(t), \beta_i(t) \leq t$. In fact, we find
\[
    \alpha_i(t)=\int_{\ux^i+t}^{\overline{x}^i+t}f_t(\xh,s)ds
    \quad\text{ and }\quad
    \beta_i(t)=-\int_{\ux^i}^{\overline{x}^i}f_t(\xh,s)ds.
\]
Note that, since $[\ux^i+t,\overline{x}^i+t]\subseteq\overline{\Omega}_t$ and $[\ux^i,\overline{x}^i]\subseteq\overline{\Omega}$, the sign of $f_t(\xh,\cdot)$ cannot change over the intervals of integration. Moreover, $\alpha_i(t)=t>0$ if and only if $\overline{x}^i-\ux^i\ge t$ and $\overline{x}^i+t\le\ux^j$ for all $\ux^j>\overline{x}^i$. Similarly, $\beta_i(t)=t>0$ if and only if $\overline{x}^i-\ux^i\ge t$ and $\ux^i\ge\overline{x}^j+t$ for all $\overline{x}^j<\ux^i$.
We see that if $t_0>0$ and $\alpha_i(t_0)=t_0$, then $\alpha_i(t)=t$ for all $0\le t\le t_0$, similarly for $\beta_i$. In particular, if $\ux_i = \overline{x}_i$ for some $i \in \Pi_{\xh}$ (isolated points) then $\alpha_i(t) = \beta_i(t)=0$, for all $t\ge0$. For $t>0$, we introduce
\[
    \Gamma^1_{\xh}(t):=\{i\in\Pi_{\xh}:\alpha_i(t)=t\}
    \quad\text{ and }\quad
    \Gamma^2_{\xh}(t):=\{i\in\Pi_{\xh}:\beta_i(t)=t\}
\]
We also define $\Gamma_{\xh}:=\Gamma^1_{\xh}\cap\Gamma^2_{\xh}$. From the oberservations above, we see that $N_k(t):=\text{card}\left(\Gamma^k_{\xh}(t)\right)$, $k=1,2$, are nondecreasing and that $N(t):=\text{card}\left(\Gamma_{\xh}(t)\right)\le\min\left\{N_1(t),N_2(t)\right\}$. Finally, we set and $\Pi^+_{\xh}=\{i\in\Pi_{\xh}:\ux^i<\overline{x}^i\}$.

We will argue that
\begin{equation}\label{eq: alphaBetaLimit}
    \sum_{i \in \Pi^+_{\xh}}\left[\alpha_i(t) + \beta_i(t)\right] \to 0 ~\text{as}~ t \to 0.
\end{equation}
Recalling that $f_t(\xh, x_n) = \lambda'(\xh,x_n-t)-\lambda'(\xh,x_n)$, we see that
\begin{equation}\label{eq: alphaBetaSum}
    \mathcal{H}^1\left(\{s: |f_t(\xh, s)|>1/2\}\right)
    = \sum_{i \in \Pi_{\xh}}[\alpha_i(t) + \beta_i(t)]
    = \sum_{i \in \Pi^+_{\xh}} [\alpha_i(t) + \beta_i(t)].
\end{equation}
The claim will follow from showing that $f_t(\xh,\cdot)\to 0$ in measure, as $t\to0^+$. To this end, set $S=\{s:|f_t(\xh,s)|>0\text{ for all }t>0\}$. If $s\in\pinv_{\Omega}(\xh)\cup\pinv_{\overline{\Omega}^c}(\xh)$, then
\[
    t'=\min\left\{\inf_{i\in\Pi_{\xh}}|\overline{x}^i-s|,
        \inf_{i\in\Pi_{\xh}}|\ux^i-s|\right\}>0,
\]
implying $f_t(\xh,s)=0$ for each $0<t<t'$. Thus, $S$ is contained in $\{\overline{x}^i:i\in\Pi_{\xh}\}\cup\{\ux^i:i\in\Pi_{\xh}\}$, which is countable and has measure zero. Since $f_t \to 0$ pointwise a.e. and there is a set of finite measure that contains $\bigcup_{0<t<1}\text{supp}\left(f_t(\xh,\cdot)\right)$, we conclude that $f_t\to0$ in measure. In view of~\eqref{eq: alphaBetaSum}, we see that~\eqref{eq: alphaBetaLimit} must be true.


To continue, recall that $0\le\alpha_i(t),\beta_i(t)\le t$, for all $i\in\Pi_{\xh}$ and $t\ge0$. Therefore,~\eqref{eq: alphaBetaLimit} implies
$$\lim_{t \to 0} t N_1(t)  \leq \lim_{t \to 0} \sum_{i \in \Pi^+_{\xh}} \alpha_i(t) =0\quad\text{ and } \quad \lim_{t \to 0} t N_2(t)  \leq \lim_{t \to 0} \sum_{i \in \Pi^+_{\xh}} \beta_i(t)=0. $$

Now, define 
\[
    g_t(\xh, s): = \sum_{i \in \Gamma_{\xh}(t)} \chi_{[\ux^i , \overline{x}^i]} (\xh, s) \quad \text{and}\quad
    g_0(\xh, s): = \sum_{i \in \Pi^+_{\xh}} \chi_{[\ux^i , \overline{x}^i]} (\xh, s).
\]
Arguing as we did to show $f_t(\xh,\cdot)\to 0$ in measure, we can verify that $g_t(\xh,\cdot)\to g_0(\xh,\cdot)$ in measure, as $t\to0^+$. In particular,
\begin{equation}\label{eq: gConvMeas}
    \lim_{t \to 0} |\{s: |g_t(\xh,s) - g_0(\xh,s)|>1/2 \}|
    = \lim_{t \to 0}  \sum_{i \in \Pi^+_{\xh} \setminus \Gamma_{\xh} (t)}|\overline{x}^i - \ux^i| = 0.
\end{equation}
With the preliminaries above, we turn to establishing~\eqref{eq: bndCurv} and~\eqref{eq: extraInt}.




First, we show~\eqref{eq: bndCurv}. We have
\begin{align}
\nonumber
    &\int_{E_t} H^J(x) dx - \int_{F_t} H^J(x) dx
    = \int_{R} \int^\infty_{-\infty} \big(H^J(\xh, s+t) - H^J(\xh,s)\big) \lambda'(\xh,s)  ds  d \xh \\
\nonumber
    &\qquad= \int_R \sum_{i \in \Pi^+_{\xh}} \int^{\overline{x}^i}_{\ux^i} \big(H^J(\xh, s+t) - H^J(\xh,s)\big)  ds d \xh \\
\label{eq: integrals1}
    &\qquad= \int_R \sum_{i \in \Gamma_{\xh}(t)} \int^{\overline{x}^i}_{\ux^i} \big(H^J(\xh, s+t) - H^J(\xh,s)\big)   ds d \xh\\
\nonumber
    &\qquad\qquad\qquad\qquad+ \int_R \sum_{i \in \Pi^+_{\xh} \setminus \Gamma_{\xh}(t)} \int^{\overline{x}^i}_{\ux^i} \big(H^J(\xh, s+t) - H^J(\xh,s)\big)  ds  d \xh.
\end{align}
For the last expression, the uniform continuity of $H^J$ implies
$$\frac{1}{t}\bigg|\int_R \sum_{i \in \Pi^+_{\xh} \setminus \Gamma_{\xh}(t)} \int^{\overline{x}^i}_{\ux^i} \big(H^J(\xh, s+t) - H^J(\xh,s)\big)  \, ds \, d \xh \bigg| \leq \int_R \sum_{i \in \Pi^+_{\xh} \setminus \Gamma_{\xh}(t)}  C(\overline{x}^i - \ux^i) \, d \xh. $$
Using~\eqref{eq: gConvMeas}, we conclude that
\[
    \lim_{t\to0^+}\frac{1}{t}\left|\int_R \sum_{i \in \Pi^+_{\xh} \setminus \Gamma_{\xh}(t)} \int^{\overline{x}^i}_{\ux^i} \left(H^J(\xh, s+t) - H^J(\xh,s)\right)  \, ds \, d \xh \right|
    =0
\]
We not turn to the first interated integral in~\eqref{eq: integrals1}. Focusing on the inner integral, we can write
\begin{align*}
    &\int^{\overline{x}^i}_{\ux^i} \big(H^J(\xh, s+t) - H^J(\xh,s)\big)  \, ds
    = \int^{\overline{x}^i + t}_{\ux^i + t} H^J(\xh, s) \, ds - \int^{\overline{x}^i}_{\ux^i} H^J(\xh, s) \, ds \\
    &\qqquad= \int^{\overline{x}^i + t}_{\overline{x}^i} H^J(\xh,s) \, ds - \int^{\ux^i + t}_{\ux^i} H^J(\xh,s) \, ds \\
    &\qqquad= \int^t_0 \big[H^J(\xh, s + \overline{x}^i)- H^J(\xh, \overline{x}^i) \big] \, ds + \int^t_0 \big[H^J(\xh, s + \ux^i)- H^J(\xh, \ux^i) \big] \, ds \\
    &\qqqquad\qqqquad\qqqquad\qquad + t \big[H^J(\xh, \overline{x}^i) - H^J(\xh, \ux^i) \big]. 
\end{align*}
The uniform continuity of $J^J$ provides the following bounds:
$$ \frac{1}{t} \left|\int^t_0 \left[H^J(\xh, s + \overline{x}^i)- H^J(\xh, \overline{x}^i) \right] \, ds\right|, \frac{1}{t} \left|\int^t_0 \left[H^J(\xh, s + \ux^i)- H^J(\xh, \ux^i)\right] \, ds \right| \leq Ct.$$
Returning to~\eqref{eq: integrals1}, we use the facts that 
$$\H^{n-1}(R) < \infty,\quad\lim_{t \to 0} t N(t) = 0,\quad\text{ and }\quad \lim_{t \to 0}  \sum_{i \in \Pi^+_{\xh} \setminus \Gamma_{\xh} (t)}|\overline{x}^i - \ux^i| = 0$$
to conclude that~\eqref{eq: bndCurv}

For the first difference of integrals in~\eqref{eq: extraInt},
we perform a very similar analysis as in the Lemma~\ref{lem: simple}. To start, we write
$$\lim_{t \to 0} \frac{1}{t} \bigg(\int_{E_t} \int_{E_t} J(x-y) \,dy \, dx - \int_{F_t} \int_{E_t} J(x-y) \, dy \, dx \bigg) =\lim_{t \to 0} \int_R \sum_{i \in \Gamma_{\xh}(t)}  \big[I^J_t(\xh, \overline{x}^i) - I^J_t(\xh, \ux^i) \big] \, d \xh,$$
where we recall that $I^J_t(x)= \int_{E_t} J(x-y) dy$. Next, observe that
\begin{equation}\label{equal3}
    \int_R \sum_{i \in \Gamma_{\xh}(t)}  \big[I^J_t(\xh, \overline{x}^i) - I^J_t(\xh, \ux^i) \big] \, d \xh 
    =\int_R \sum_{i \in \Gamma_{\xh}(t)}  (\overline{x}^i - \ux^i) \frac{\big[I^J_t(\xh, \overline{x}^i) - I^J_t(\xh, \ux^i) \big]}{(\overline{x}^i - \ux^i) } d \xh. 
\end{equation}
For a fixed $t>0$, if $i \in \Gamma_{\xh}(t)$, then $\overline{x}^i - \ux^i \geq t$. Performing similar computations for $I^{J}_t$ as in~\eqref{eq:HLip} for $H^J$ (under either assumption (J4a) or (J4b)), we have
$$ \bigg|\frac{I^J_t(\xh, \overline{x}^i) - I^J_t(\xh, \ux^i) }{\overline{x}^i - \ux^i } \bigg| \leq C,$$
and hence, for all $t>0$ and $\xh \in R$,
$$\sum_{i \in \Gamma_{\xh}(t)}  (\overline{x}^i - \ux^i) \bigg|\frac{\big[I^J_t(\xh, \overline{x}^i) - I^J_t(\xh, \ux^i) \big]}{(\overline{x}^i - \ux^i) } \bigg| \leq  C \sum_{i \in \Pi^+_{\xh}}(\overline{x}^i - \ux^i) \leq C \diam(\Omega).$$
Notice that $|E_t| = \int_R \sum_{i \in \Pi^+_{\xh}} \alpha_i(t) \,  d \xh$, and we proved that 
$\lim_{t \to 0}\sum_{i \in \Pi_{\xh}} \alpha_i(t) = 0 $. Thus, for each $\xh\in R$ and $i \in \Pi^+_{\xh}$, 
$$ \lim_{t \to 0} \frac{I^J_t(\xh, \overline{x}^i) - I^J_t(\xh, \ux^i) }{\overline{x}^i - \ux^i } = 0.$$
Using the Dominated Convergence Theorem in \eqref{equal3}, we deduce that
$$\lim_{t \to 0} \int_R \sum_{i \in \Gamma_{\xh}(t)}  \big[I^J_t(\xh, \overline{x}^i) - I^J_t(\xh, \ux^i) \big] \, d \xh =0. $$
The same argument also applies to $\int_{F_t} \int_{F_t} J(x-y) \,dy \, dx - \int_{E_t} \int_{F_t} J(x-y) \, dy \, dx$.

In conclusion, we have shown that
$$0 =  \lim_{t \to 0} \frac{1}{t} \bigg(P^J(\Omega_t) - P^J(\Omega) \bigg) = \int_R \sum_{i \in \Pi_{\xh}}  \big[H^J(\xh, \overline{x}^i) - H^J(\xh, \ux^i) \big] \, d \xh.$$
Under the ordered curvature assumption, $H^J(\xh, \overline{x}^i) {\geq} H^J(\xh, \ux^i)$.  Therefore, $H^J(\xh, \overline{x}^i) = H^J(\xh, \ux^i)$, for each $i \in \Pi_{\xh}$ and $\xh \in R$, as needed.
\end{proof}

\section{The Moving Plane Method}
Next we introduce a generalization of Alexandrov's moving plane method in the context of nonlocal ordered curvature. The celebrated moving plane method was introduced by Alexandrov in 1960 (\cite{Alexandrov}) to prove the fact that a bounded $C^2$ surface with constant mean curvature is a sphere. Later, it was used by many others to prove similar symmetry results. Thus, it was employed to show that a bounded connected set whose  surface has constant nonlocal mean curvature is a ball \cite{Bucur1, Bucur2, Cabre-Fall-Sola-Weth, Figalli}. The regularity of the surface in these works varies from $C^{1,\alpha}$ to simply measurability, depending on the singularity of the kernel. For this method, one takes the reflection of $\Omega$ with respect to a hyperplane $x_n = \lambda$ for some $\lambda \in \R$; we denote the reflection by $R_\lambda(\Omega)$. The hyperplane $x_n = \lambda$ is labeled $\HP_\lambda$. We denote by $\HP^+_\lambda = \{(\xh, x_n): x_n > \lambda  \}$, $\Omega^+_\lambda = \Omega \cap \HP^+_\lambda$, $R^+_\lambda(\Omega) = R_{\lambda}(\Omega) \cap \HP^+_\lambda$.
We may assume $\overline{\Omega}\subseteq\HP^+_{0}$.
\begin{defn}\label{D:critical_plane}
    $\HP_{\lambda_0}$ is called a critical plane if the following hold 
    \begin{itemize}
          \item For all $0<\lambda< \lambda_0$, we have $\ROmCl \subseteq \Omega$.
        \item There exists a decreasing sequence $\lambda_k \to \lambda_0$ such that for all $k \in \mathbb{N}$, $\ROmClk \setminus \Omega \neq \emptyset$. 
    \end{itemize}
\end{defn}
 Without loss of generality we can assume that $\lambda_0 = 0$. For an arbitrary point $x$ we denote by $\tx$ its reflection in the positive  half-space $\HP^+_0$.
 \begin{lem}
We have    $R^+_0(\overline{\Omega}) \setminus \Omega = R^+_{0}(\partial \Omega) \cup \partial \Omega $.
\end{lem}
\begin{proof}
    We only need to show that $R^+_0(\overline{\Omega}) \setminus \Omega \subseteq R^+_{0}(\partial \Omega) \cup \partial \Omega$. If $R^+_0(\overline{\Omega}) \setminus \Omega = \emptyset$, we are done. Otherwise, we notice: 
    $$ \ROmlO \setminus \Omega =  \big(\ROmlO \setminus \overline{\Omega} \big) \cup \big(\ROmlO \cap \partial{\Omega} \big) = \big(\ROmlO \setminus \overline{\Omega} \big) \cup \big(R^+_0(\Omega) \cap \partial{\Omega} \big) \cup  \big(R^+_{0}(\partial \Omega) \cup \partial \Omega \big).$$
    For $\tx =(\widehat{\tx},\tx_n) \in \ROmlO \setminus \overline{\Omega}$ there exists $\lambda_1<0$ such that $(\widehat{\tx}, \tx_n + 2\lambda_1)\in \ROmlO \setminus \overline{\Omega}$. As $x= (\widehat{\tx},-\tx_n)$, we have $R_0(x) = \tx$. However, $(\widehat{\tx}, \tx_n + 2\lambda_1) = R_{\lambda_1}(x)$, which contradicts the fact that $\lambda =0$ is the critical hyperplane. 

    On the other hand, if $\tx \in \big(R^+_0(\Omega) \cap \partial{\Omega} \big)$, then there exists $x=(\widehat{\tx},-\tx_n) \in \Omega \cap \HP^-_0$ such that $R_0(x) = \tx$. Since $\Omega \cap \HP^-_0$ is an open set, we can choose $\lambda_1<0$, such that $(\widehat{\tx}, -\tx_n + 2 \lambda_1) \in \Omega \cap \HP^-_0$ which further implies $R_{\lambda_1}(\widehat{\tx}, -\tx_n + 2 \lambda_1) = \tx$. This contradicts the fact that $\lambda=0$ is the critical hyperplane, since $R_{\lambda_1}(\widehat{\tx}, -\tx_n + 2 \lambda_1) \in \ROmlO \setminus \Omega$.
 \end{proof}
Based on this lemma, we consider the following two cases:
\begin{enumerate}[(a)]
    \item Whenever $\tx \in \ROmlO \setminus \Omega = R^+_0(\partial \Omega) \cap \partial \Omega$, we say that $\tx$ is an interior touching point.
    \item On the other hand, if $\ROmlO \setminus \Omega = \emptyset$ then we have a non-transversal intersection.
\end{enumerate}

 \begin{lem}
     Consider $\{\lambda_k\}$ to be the decreasing sequence such that $\lambda_k \to 0$ and for all $k \in \mathbb{N}$, $\ROmClk \setminus \Omega \neq \emptyset$.
     If there exists $\lambda^* >0$ and a subsequence $\{\lambda_{k_l} \} \subseteq \{\lambda_{k} \}$ such that for all $l \in \mathbb{N}$,
     $$ \ROmClkl \setminus (\HP^-_{\lambda^*} \cup \Omega) \neq \emptyset,$$
     then there is an interior touching point on $\partial \Omega$.
 \end{lem}
\begin{proof}
From the given assumption, we have for every $l \in \mathbb{N}$, there exists $\Tilde{x}_l =(\Tilde{\xh}_l, \Tilde{x}_{l,n}) \in \ROmClkl \setminus \Omega$ such that $\Tilde{x}_{l,n} \geq \lambda^*$. Since $\cup_{l=1}^\infty R_{\lambda_{k_l}}(\overline{\Omega}) \subseteq \R^n$ is a bounded set, there is a subsequence (using the same index $l$) $\Tilde{x}_l$ such that $\lim_{l \to \infty} \Tilde{x}_l =\Tilde{x}_0$. We notice that $\Tilde{x}_0 \in \Omega^c$ and $\Tilde{x}_{0,n} \geq \lambda^*$. Next we see that for those $l \in \mathbb{N}$, there exists $x_l = (\xh_l, x_{l,n}) \in \overline{\Omega}$ such that $x_{l,n} = 2 \lambda_{k_l} - \Tilde{x}_{l,n}$. Since $\lambda_{k_l} \to 0$ as $l \to \infty$, and $\overline{\Omega}$ is closed, $x_l$ has a converging subsequence in $\overline{\Omega}$. Let that sequence converge to $x_0 \in \overline{\Omega}$. As we can write that $x_{0,n} = 0-\Tilde{x}_{0,n}$ we have that $x_0 \in \HP^-_0$, hence $\Tilde{x}_0 \in R^+_0(\overline{\Omega})$. Thus, we conclude that $\Tilde{x}_0 \in \R^+_0(\overline{\Omega}) \setminus \Omega$.
\end{proof}

\begin{lem}\label{lem:interior_constant_kernel}
    Suppose there exists $x_0 \in \overline{\Omega} \cap \HP^-_0$ such that $\Tilde{x}_0 = R_0(x_0) \in R^+_0(\overline{\Omega}) \setminus \Omega$, {where $\Omega$ satisfies assumption \ref{assum_domain}}  and $J(x) = \chi_{B_r}(x)$ (assumption (J4a)), then
    $$|(\Omega \setminus R_0(\Omega)) \cap (B_r(\Tilde{x}_0)\setminus B_r(x_0))| =  |(  R_0(\Omega) \setminus \Omega) \cap (B_r(\Tilde{x}_0) \setminus B_r(x_0))|= 0.$$
\end{lem}
\begin{proof}
    Using the Definition \ref{D:critical_plane} of the critical plane and previous lemmas, we see that $x_0 \sim \Tilde{x}_0$. Otherwise, there would be a point on the line segment joining $x_0$ and $\Tilde{x}_0$ that does not belong to $\overline{\Omega}$, which would contradict the fact that $\lambda=0$ is the critical hyperplane. Hence, $H^J(x_0) = H^J(\Tilde{x}_0)$, {due to Theorem \ref{thm:monotone_implies_equality}}. Then, we have
\begin{align}\label{eq:interior_touching}
        0 &= H^J(R_0(x_0)) - H^J(x_0)
        = H^J_{R_0(\Omega)}(x_0) - H_\Omega^J(x_0) \\ \nonumber
        &= 2 \int_{\Omega \setminus R_0(\Omega)} J(x_0-y) dy -2 \int_{R_0(\Omega) \setminus \Omega} J(x_0 -y) dy \\ \nonumber
        &= 2 \int_{\Omega \setminus R_0(\Omega)} \big(J(x_0-y)-J(x_0 - R_0(y)) dy \\ \nonumber
        &= -2|(\Omega \setminus R_0(\Omega))\cap (B_r(\tx_0) \setminus B_r(x_0)|. 
    \end{align}
    Hence, $|(\Omega \setminus R_0(\Omega))\cap (B_r(\tx_0) \setminus B_r(x_0)|=0$ and similarly $|(R_0(\Omega) \setminus \Omega)\cap (B_r(x_0) \setminus B_r(\tx_0)|=0$.
    Thus, we have our claim. 
\end{proof}

\begin{lem}
    Suppose that there exists $x_0 \in \overline{\Omega} \cap \HP^-_0$ such that $\Tilde{x}_0 = R_0(x_0) \in R^+_0(\overline{\Omega}) \setminus \Omega$, {where $\Omega$ satisfies assumption \ref{assum_domain}} and (J4b) holds. Then
    $$|(\Omega \setminus R_0(\Omega)) \cap B_r(\tilde{x}_0)| =  |(  R_0(\Omega) \setminus \Omega) \cap B_r(x_0)|= 0.$$
\end{lem}
\begin{proof}
    Performing a similar computation as in the proof of the previous Lemma \ref{lem:interior_constant_kernel} we get that $$ 0 = 2 \int_{\Omega \setminus R_0(\Omega)} \big(J(x_0-y)-J(x_0 - R_0(y)) dy.$$
    Then we observe that for every $y \in \Omega \setminus R_0(\Omega)$, $|x_0-y| > |x_0 - R_0(y)|$ 
    and then using the fact that $\mu$ is strictly decreasing on its support set, we have that for all $y \in (\Omega \setminus R_0(\Omega)) \cup B_r(\tilde{x}_0)$:
    $$ J(x_0-y) = \mu(|x_0-y|) < \mu(|x_0-R_0(y)|) = J(x_0 - R_0(y)).$$
    This implies:
    $$|(\Omega \setminus R_0(\Omega)) \cap B_r(\tilde{x}_0)| = 0.$$
    We can prove the other claim (i.e. $|(  R_0(\Omega) \setminus \Omega) \cap B_r(x_0)|= 0$) in a similar way.
\end{proof}

\begin{lem}\label{lem: interior}
    If there is an interior touching point $\Tilde{x}_0 \in R^+_0(\overline{\Omega}) \setminus \Omega$, {where $\Omega$ satisfies assumption \ref{assum_domain}}, then $R_0(\overline{\Omega}) = \overline{\Omega}$. 
   
\end{lem}
\begin{proof}
We will show the result under assumption (J4a); i.e., $J(x)= \chi_{B_r}(x)$. The result can be similarly proved for kernels satisfying (J4b).
     We start by arguing
    $$\partial \Omega \cap B_r(\Tilde{x}_0) \setminus \overline{B_r({x}_0)} = R_0(\partial \Omega) \cap B_r(\Tilde{x}_0) \setminus \overline{B_r({x}_0)}$$
    and
    $$\partial \Omega \cap B_r(x_0) \setminus \overline{B_r(\Tilde{x}_0)} = R_0(\partial \Omega) \cap B_r(x_0) \setminus \overline{B_r(\Tilde{x}_0)}.$$
    This proof is quite similar to the computations given in \cite{Bucur1}. Using Lemma \eqref{lem:interior_constant_kernel},  we have 
    $$|(\Omega \setminus R_0(\overline{\Omega}))\cap (B_r(\tx_0) \setminus \overline{B_r(x_0)}| = |(R_0(\Omega) \setminus \overline{\Omega})\cap (B_r(x_0) \setminus \overline{B_r(\tx_0)}| =0.$$
    This implies $(\Omega \setminus R_0(\overline{\Omega}))\cap (B_r(\tx_0) \setminus \overline{B_r(x_0)} = \emptyset$, since $\Omega$ is an open set.
    Next, we choose a point $\Tilde{y} \in (\partial \Omega \setminus R_0(\partial \Omega)) \cap (B_r(\Tilde{x}_0) \setminus \overline{B_r({x}_0)}) \cap \HP^+_0$.   Since $B_r(\Tilde{x}_0) \setminus \overline{B_r({x}_0)}$ is an open set, we have $\epsilon>0$ such that $B_{2 \epsilon}(\Tilde{y}) \subseteq (B_r(\Tilde{x}_0) \setminus \overline{B_r({x}_0)}) \cap \HP^+_0$. Now we assume that there exists a sequence $\{\Tilde{y}^k\} \in \Omega \cap B_\epsilon(\Tilde{y})$ such that $\Tilde{y}^k \to \Tilde{y}$. Since $(\Omega \setminus R_0(\overline{\Omega}))\cap (B_r(\tx_0) \setminus \overline{B_r(x_0)} = \emptyset$, we have $\{\Tilde{y}^k\} \subseteq R_0(\overline{\Omega}) \cap B_\epsilon(\Tilde{y})$. Then we have $\Tilde{y} \in R_0(\overline{\Omega}) \cap \HP^+_0$. However, that would imply $\Tilde{y} \in R^+_0(\partial \Omega) \cap \partial \Omega$, a contradiction. Hence, we have $$\partial \Omega \cap B_r(\tilde{x}_0) \setminus \overline{B_r({x}_0)} = R_0(\partial \Omega) \cap B_r(\tilde{x}_0) \setminus \overline{B_r({x}_0)}.$$ This implies that every boundary point in $B_r(\tilde{x}_0) \setminus \overline{{B_r}(x_0)}$ is a touching point. We also observe that the set of touching points is open with respect to $\partial \Omega \cap \HP^+_0$.
    Similarly, we can prove the other claim. 
    Next we observe that the set of touching points is $\partial \Omega \cap R_0(\partial \Omega) \cap HP^+_0$ and hence it is closed in $\partial \Omega \cap HP^+_0$. Then $R_0(\Omega) = \Omega$.
    
\end{proof}

\begin{figure}
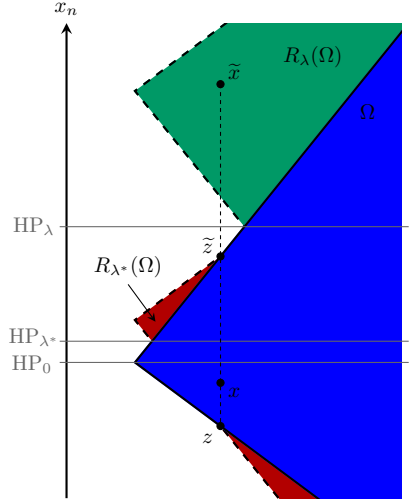

\centering
\tikz[scale=.9,every node/.style={scale=.8}]{

\begin{scope}
\clip (0,0)rectangle(4.9,7);
\draw[-,thick,dashed,blue!40!green] (5,1)--(1,6)--(5,9)--cycle;
\draw[-,thick,dashed,fill=blue!40!green,opacity=.2] (5,1)--(1,6)--(5,9)--cycle;
\draw[-,thick,dashed,red!70!black] (5,-2.375)--(1,2.625)--(5,5.625)--cycle;
\draw[-,thick,dashed,fill=red!70!black,opacity=.2] (5,-2.375)--(1,2.625)--(5,5.625)--cycle;
\draw[-,thick,blue!70] (5,-1)--(1,2)--(5,7)--cycle;
\draw[-,thick,fill=blue,opacity=.3] (5,-1)--(1,2)--(5,7)--cycle;
\end{scope}

\draw[-stealth,thick] (0,0)--(0,7) node[above]{$x_n$};
\draw[-,thin,black!60] (5,2)--(-.1,2) node[left,yshift=-1]{$\HP_0$};
\draw[-,thin,black!60] (5,4)--(-.1,4) node[left,xshift=1.3]{$\HP_{\lambda}$};
\draw[-,thin,black!60] (5,2.3125)--(-.1,2.3125) node[left,xshift=3.5,yshift=1.8]{$\HP_{\lambda^*}$};
\draw[fill=black] (2.25,1.7)circle(.052) node[below right,yshift=1.3]{$x$};
\draw[fill=black] (2.25,6.1)circle(.052) node[above right,yshift=-1.3]{$\widetilde{x}$};
\draw[fill=black] (2.25,1.0625)circle(.052) node[below left]{$z$};
\draw[fill=black] (2.25,3.5625)circle(.052) node[above left,yshift=-2]{$\widetilde{z}$};
\draw[-,thin,dashed,dash pattern=on 1.5pt off 1.5pt] (2.25,6.1)--(2.25,1.0625);
\node at(4.4,5.7){$\Omega$};
\node at (3.6,6.5){$R_\lambda(\Omega)$};
\draw[stealth-,thin] (1.3,2.6)--(.9,3.15) node[above]{$R_{\lambda^*}(\Omega)$};
}
\caption{Identification of $z$ and $\widehat{z}$ in Case 1 for the proof of Lemma~\ref{lem:3lambda}}
\label{fig:non_trans_1}
\end{figure}


\begin{lem} \label{lem:3lambda}
    Suppose $\HP_0$ is the critical plane and $R^+_0(\overline{\Omega}) \setminus \Omega = \emptyset$. Then for each $\lambda>0$ and $\Tilde{x} \in \ROmCl \setminus \Omega,$ there exists $0<\lambda^*<2 \lambda$ and $z \in \partial \Omega \cap \HP^-_{\lambda^*}$ such that 
    \begin{itemize}
        \item $\Tilde{z}= R_{\lambda^*}(z) \in \partial \Omega$
        \item $- 3\Tilde{x}_n < z_n < \Tilde{z}_n < 3\Tilde{x}_n$  ($n$-th coordinate)
    \end{itemize}
\end{lem}

\begin{proof}

We see that 
$$(\ROmCl \setminus \Omega) = (\RPmL \setminus \overline{\Omega}) \cup (R^+_\lambda(\Omega) \setminus {\Omega}) .$$
Then we consider the following two cases:
\begin{itemize}
    
    \item[] \underline{Case $1$}: $\tx \in R^+_\lambda(\Omega) \setminus  \Omega$\\
    Then $x=(\xh, x_n) \in \Omega \setminus R_\lambda(\Omega) \cap \HP^{-}_0$. Since $x \in \Omega$, then we can define,
\begin{align*}
    z_n =& \inf\left\{ \xi < x_n :\left\{\xh\right\} \times [\xi, x_n]  \in \Omega \right\}
\intertext{and}
    \Tilde{z}_n =& \sup \left\{ \xi > x_n : \left\{\xh\right\} \times [x_n, \xi] \in \Omega \right]\}.
\end{align*}
    Then $z= (\xh, z_n) \in \partial \Omega$ and $\Tilde{z} = (\xh, \Tilde{z}_n) \in \partial \Omega$ and $z_n < x_n < \Tilde{z}_n \leq \Tilde{x}_n$,  see Figure \ref{fig:non_trans_1}. Next we define $\lambda^* = \frac{1}{2}(z_n + \Tilde{z}_n) \leq \lambda$. If $z_n \geq 0$, then $\lambda^* >0$. Otherwise, since $\HP_0$ is the critical hyperplane and there is no interior touching point, we have $\Tilde{x}_n > [R(z)]_n$ and hence $\lambda^* >0$.

\item[] \underline{Case $2$}: $\tx \in (\RPmL \setminus \overline{\Omega})$\\
    Then we have $x=(\xh, x_n) \in \partial \Omega \cap \HP^-_\lambda \setminus R_\lambda(\overline{\Omega})$ such that $R_\lambda(x) = \tx$. We have two subcases:
\begin{itemize}
\item[] \underline{Subcase 2a}: $x_n < 0$\\
Then $R_0(x) \in \overline{\Omega}$. As in Case 1, we define,
\[
    \Tilde{z}_n = \sup \left\{ \xi > x_n : \left\{\xh\right\} \times [\xi, z^*_n]  \in \overline{\Omega} \right\}.
\]
    Setting $\Tilde{z}=(\xh, \Tilde{z}_n)$ and $z=x$, we can define $\lambda^* = \frac{(z_n + \Tilde{z}_n)}{2}$. Then we have,
    $-\Tilde{x}_n < x_n =z_n < -x_n <  \Tilde{z}_n \leq \Tilde{x}_n$, $0<\lambda^* \leq \lambda$ and $z \sim \Tilde{z}$.
\item[] \underline{Subcase 2b}: $0 \leq x_n < \lambda$\\
    Since $\tx \in \R^n \setminus \overline{\Omega}$, there exists an $0<\epsilon<\lambda$ such that $B_\epsilon(\tx) \subset \R^n \setminus \overline{\Omega}$. Thus, we may select $z^* = (\widehat{z}^*, z^*_n) \in B_\epsilon(x) \cap \Omega$.     
    Let us define,
\[
    \Tilde{z}_n = \sup \left\{ \xi > z^*_n : \left\{\xh\right\} \times [\xi, z^*_n]  \in\Omega \right\},
\]
    which gives us $\Tilde{z}=(\widehat{z}^*, \Tilde{z}_n) \in \partial \Omega$. 
    Then we have $x_n-\epsilon < z^*_n < \Tilde{z}_n\le \Tilde{x}_n \le 2 \lambda$.
    Similarly, we define
\[
    z_n = \inf \left\{ \xi < z^*_n : \left\{\xh\right\} \times [\xi, z^*_n]  \in \Omega \right\}
\]
    and $\lambda^* = (z_n + \Tilde{z}_n)/2 < \Tilde{z}_n < 3 \lambda$. As $\HP_0$ is the critical hyperplane, we must have $\lambda^*>0$. Thus, $\Tilde{z}= R_{\lambda^*}(z)$ and $-3 \lambda < 2 \lambda^* - 3 \lambda < 2 \lambda^* - \Tilde{z}_n = z_n < \Tilde{z}_n < 3 \lambda$.
\end{itemize}

\end{itemize}

\end{proof}

From the construction of these points, again we notice that $z \sim \Tilde{z}$.
\begin{figure}[h]
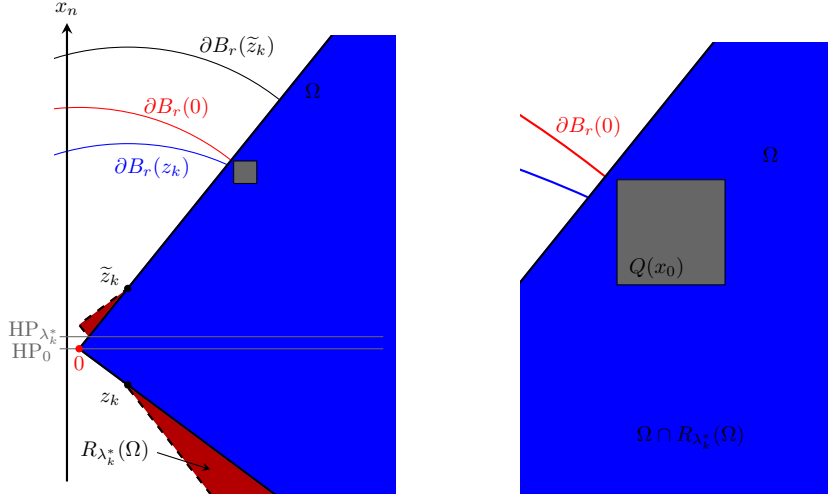

\centering
\parbox[b]{.4\textwidth}{}
\tikz[scale=3.2,every node/.style={scale=.8}]{

\begin{scope}
\clip (0.9,1.4)rectangle(2.3,3.3);
\draw[-,red] (1,2)circle(1);
\draw[-,black] (1.2,2.25)circle(1);
\draw[-,blue] (1.2,1.85)circle(1);
\node[red] at(1.4,3){$\partial B_r(0)$};
\node[black] at(1.65,3.25){$\partial B_r(\widetilde{z}_k)$};
\node[blue] at(1.3,2.75){$\partial B_r(z_k)$};
\draw[-,thick,dashed,red!70!black] (5,-3.1)--(1,2.1)--(5,5.1)--cycle;
\draw[-,thick,dashed,fill=red!70!black,opacity=.2] (5,-3.1)--(1,2.1)--(5,5.1)--cycle;
\draw[-,thick,blue!70] (5,-1)--(1,2)--(5,7)--cycle;
\draw[-,thick,fill=blue,opacity=.3] (5,-1)--(1,2)--(5,7)--cycle;
\end{scope}

\draw[-stealth,thick] (0.95,1.45)--(0.95,3.35) node[above]{$x_n$};
\draw[-,thin,black!60] (2.25,2)--(.92,2) node[left,yshift=-1.7]{$\HP_0$};
\draw[-,thin,black!60] (2.25,2.05)--(.92,2.05) node[left,xshift=3.5,yshift=1.7]{$\HP_{\lambda^*_k}$};
\draw[red,fill=red] (1,2)circle(.013) node[below]{$0$};
\draw[fill=black] (1.68,2.7334)circle(.012);
\draw[fill=black!60,opacity=.3](1.635,2.685)rectangle(1.73,2.7782);
\draw[fill=black] (1.2,1.85)circle(.013) node[below left]{$z_k$};
\draw[fill=black] (1.2,2.25)circle(.013) node[above left,yshift=-2]{$\widetilde{z}_k$};
\node at (1.96,3.07){$\Omega$};
\draw[stealth-] (1.54,1.5)--(1.32,1.57) node[left]{$R_{\lambda_k^*}(\Omega)$};
}
\hspace{40pt}
\parbox[b]{.35\textwidth}{
\tikz[scale=14.95,every node/.style={scale=.8}]{

\begin{scope}
\clip (1.55,2.5)rectangle(1.82,2.9);
\draw[-,red,thick] (1,2)circle(1);
\draw[-,black,thick] (1.2,2.25)circle(1);
\draw[-,blue,thick] (1.2,1.85)circle(1);
\node[red] at(1.61,2.825){$\partial B_r(0)$};
\draw[-,thick,dashed,red!70!black] (5,-3.1)--(1,2.1)--(5,5.1)--cycle;
\draw[-,thick,dashed,fill=red!70!black,opacity=.2] (5,-3.1)--(1,2.1)--(5,5.1)--cycle;
\draw[-,thick,blue!70] (5,-1)--(1,2)--(5,7)--cycle;
\draw[-,thick,fill=blue,opacity=.3] (5,-1)--(1,2)--(5,7)--cycle;

\draw[-stealth,thick] (0.95,1.45)--(0.95,3.35) node[above]{$x_n$};
\draw[-,thin,black!60] (2.05,2)--(.92,2) node[below left,yshift=1.3]{$\pi_0$};
\draw[-,thin,black!60] (2.05,2.05)--(.92,2.05) node[left,xshift=3.5,yshift=.5]{$\pi_{\lambda^*_k}$};
\draw[red,fill=red] (1,2)circle(.013) node[below]{$0$};
\draw[fill=black] (1.68,2.7338)circle(.003) node[above right]{$x_0$};
\draw[fill=black!60,opacity=.3](1.635,2.685)rectangle(1.73,2.7782);
\node at (1.67,2.7){$Q(x_0)$};
\draw[fill=black] (1.2,1.85)circle(.013) node[below left]{$z_k$};
\draw[fill=black] (1.2,2.25)circle(.013) node[above left,yshift=-2]{$\widetilde{z}_k$};
\node at (1.77,2.8){$\large\Omega$};
\node at (1.7,2.55) {$\large\Omega\cap R_{\lambda_k^*}(\Omega)$};
\end{scope}

}
}
\caption{(Left) $k$-th step in proof of Lemma~\ref{lem:non_trans}. (Right) The fixed cylinder $Q(x_0)\subseteq \Omega\cap\HP_0^+\setminus R_0(\overline{\Omega})$, with $x_0\in\partial B_r(0)$.}
\label{fig:non_trans_2}
\end{figure}

\begin{lem}\label{lem:non_trans}
    Suppose $\HP_0$ is a nontransveral plane and assumption (J4a) holds. {In addition to that, $\Omega$ satisfies assumption \ref{assum_domain}}. Then
    $$\Omega \cap \partial B_r(0) \setminus R_0(\overline{\Omega}) = \emptyset, \quad R_0(\Omega) \cap \partial B_r(0) \setminus \overline{\Omega} =\emptyset. $$
\end{lem}
\begin{proof}
    We will show that $\Omega \cap \HP^+_0 \cap \partial B_r(0) \setminus R_0(\overline{\Omega}) = \emptyset$. Suppose there exists $x_0 \in \Omega \cap \HP^+_0 \cap \partial B_r(0) \setminus R_0(\overline{\Omega})$. Since $\Omega \cap \HP^+_0 \setminus R_0(\overline{\Omega})$ is an open set, there exists an open cylinder $Q(x_0) = B'_\epsilon(\xh_0) \times (x_{0,n}-l_0, x_{0,n}+l_0) \subseteq \Omega \cap \HP^+_0 \setminus R_0(\overline{\Omega})$ for some $\epsilon, l_0 >0$, see the Figure \ref{fig:non_trans_2}.

    Since $\HP_0$ is the critical hyperplane with non-transversal intersection, we know that there exists a sequence $\lambda_k \to 0$ and a sequence $\Tilde{x}_k$ such that $\Tilde{x}_k \in \ROmClk \setminus \Omega$. By Lemma \ref{lem:3lambda}, we have $\lambda^*_k, z_k, \Tilde{z}_k$ such that 
    $0< \lambda^*_k < 3 \lambda_k$ and $- \Tilde{x}_{k,n} < z_{k,n} < \lambda^*_k < \Tilde{z}_{k,n} < \Tilde{x}_{k,n}$. Then for each $k$, we have $\mu_k \geq \lambda_k$ such that $\mu_k \to 0$ and $R^+_{\lambda^*_k}(\overline{\Omega}) \setminus (\Omega \cup \HP^-_{\mu_k}) = \emptyset$. If this is not true then for every sequence of $\{ \mu_k \}$ such that $\mu_k \geq \lambda_k$ and $\mu_k \to 0$, there exists a subsequence $\mu_{k_l}$ such that $R^+_{\lambda^*_{k_l}}(\overline{\Omega}) \setminus (\Omega \cup \HP^-_{\mu_{k_l}}) \neq \emptyset$. Next we choose $\mu^*$ such that $\mu^* \geq \mu_{k_l}$ for all $l$ except for finitely many. Then we see that 
    $R^+_{\lambda^*_{k_l}}(\overline{\Omega}) \setminus (\Omega \cup \HP^-_{\mu^*}) \neq \emptyset$ for all those $l$. That implies there exists a touching point. Hence we have a sequence $\{\mu_k\}$ such that $\mu_k \geq \lambda_k$ and $\mu_k \to 0$ and $R^+_{\lambda^*_k}(\overline{\Omega}) \setminus (\Omega \cup \HP^-_{\mu_k}) = \emptyset$. Then for all $k$, we have $\overline{\Omega} \setminus R_{\lambda^*_k}(\Omega) \subseteq \HP^+_{-\mu_k}$. Next we denote,
    $$
    \begin{cases}
        P_k = \HP^-_{\lambda^*_k} \cap \HP^+_{-\mu_k}, \\
        C_k = \HP^+_{\lambda^*_k} \cap B_r(\Tilde{z}_k) \setminus B_r(z_k), \\
        S_k = \cup_{\xh \in B_\epsilon(\xh_0)} L_k (\xh),
    \end{cases}
    $$
    where $L(\xh) = \left\{\xh\right\} \times (x_{0,n}-l_0, x_{0,n}+l_0)$ and $L_k(\xh) = L(\xh) \cap C_k$.

     Without loss of generality we assume that $0<\lambda^*_k < \mu_k< x_{0,n}- l_0$. Since $R^+_{\lambda^*_k} (\overline{\Omega}) \setminus (\Omega \cup \HP^-_{\mu_k})=\emptyset$, then we have $z_{k,n}, \Tilde{z}_{k,n} < \mu_k$. That further implies $S_k \subseteq (B_r(\Tilde{z}_k) \setminus B_r(z_k)) \cap \HP^+_{\mu_k}$. We see that for large $k$, $\mathcal{H}^1(L_k(\xh)) =( \Tilde{z}_{k,n} - z_{k,n})$. Then by Cavalieri's Principle, we have 
    $$ |S_k| = ( \Tilde{z}_{k,n} - z_{k,n}) \mathcal{H}^{n-1}(B'_r) =  \sigma r^{n-1} l_k,$$ 
    where $\sigma = \mathcal{H}^{n-1}(B'_1)$ and $l_k = ( \Tilde{z}_{k,n} - z_{k,n})$.

    Next, comparing the curvature at the boundary points $z_k, \Tilde{z}_k$, {i.e, $H_J(z_k) = H_J(\Tilde{z}_k)$ due to Theorem \ref{thm:monotone_implies_equality}} we have 
    \begin{align*}
        0 &= H^J(\Tilde{z}_k) - H^J(z_k)= H^J(R_{\lambda^*_k}(z_k)) - H^J(z_k)\\
        &=2 \int_{\Omega \setminus R_{\lambda^*_k}(\Omega)}J(|y-z_k|) dy - 2 \int_{ R_{\lambda^*_k}(\Omega)  \setminus \Omega}J(|y-z_k|) dy.
    \end{align*}
    This implies 
    $$ |(\Omega \setminus R_{\lambda^*_k}(\Omega)) \cap (B_r(z_k) \setminus B_r(\Tilde{z}_k)) | =  |(\Omega \setminus R_{\lambda^*_k}(\Omega)) \cap (B_r(\Tilde{z}_k) \setminus B_r(z_k)) |.$$
    We see that $S_k \subseteq \Omega \setminus R_{\lambda^*_k}(\Omega)) \cap (B_r(\Tilde{z}_k) \setminus B_r(z_k))$ and $(\Omega \setminus R_{\lambda^*_k}(\Omega)) \cap (B_r(z_k) \setminus B_r(\Tilde{z}_k)) \subseteq (B_r(z_k) \setminus B_r(\Tilde{z}_k)) \cap P_k$. Thus, we have
    $$ \sigma r^{n-1} l_k \leq  |(B_r(z_k) \setminus B_r(\Tilde{z}_k)) \cap P_k|.$$ 
    For each $-\mu_k < d < \lambda^*_k$, define 
    \[
        \xi(d): = \mathcal{H}^{n-1} \big( \HP_d \cap B_r(z_k) \setminus B_r(\Tilde{z}_k) \big)
        = \big|\H^{n-1}\big(B'_{r_d}(\widehat{z}_k)\big)- \H^{n-1}\big(B'_{\Tilde{r}_d}(\widehat{\Tilde{z}}_k )\big)\big|
    \]
    where $r_d^2 = r^2 - (d-z_{k,n})^2$ and $\Tilde{r}^2_d = r^2 - (d- \Tilde{z}_{k,n})^2$. For very small $\mu_k, \lambda^*_k$ we have 
    $\xi(d) \leq C_2 \sigma (r^{n-2}_d + \Tilde{r}^{n-2}_d) |r_d - \Tilde{r}_d| \sim C_2 \sigma r^{n-2} l_k$. Then we have, 
    $$ \sigma r^{n-1} l_k \leq C_2(\mu_k + \lambda^*_k)\sigma r^{n-2} l_k, $$
    which further implies that $r \leq C (\mu_k + \lambda^*_k)$. This is a contradiction, hence our claim is proved.

\end{proof}

\begin{lem}\label{lem: J4aNontran}
     Suppose $\HP_0$ is a critical hyperplane such that it is non-transversal plane and assumption (J4a) holds. Then there exists an interior touching point on $\partial \Omega \cap R^+_0(\partial \Omega)$.
\end{lem}
\begin{proof}
    By the previous lemma, we have $\Omega \cap \partial B_r(0) \setminus R_0(\overline{\Omega}) = \emptyset$, and hence $$\Omega \cap \partial B_r(0) \setminus R^+_0(\overline{\Omega}) = \emptyset.$$ 
    Since $\HP_0$ is a non-transversal intersection plane, we also have $R^+_0(\overline{\Omega}) \subseteq \Omega$. Then following situations can not happen,
    \begin{itemize}
        \item If $\Omega^+, R^+_0(\overline{\Omega}) \subseteq B_r(0)$, we notice that $\Omega \subseteq B_r(0)$, $\H^{n-1}(\Omega \cap \partial B_r(0)) = 0$. But that would contradict the non-degeneracy condition.
        \item $\partial \Omega \cap \HP^+_0$ can not intersect $\partial B_r(0)$. If it intersects, since $\Omega$ is an open set, we have a $x_0 \in (\Omega \setminus R^+_0(\overline{\Omega})) \cap \partial B_r(0)$ which contradicts our previous lemma. We also see that interior touching can't happen due to the fact $R^+_0(\Omega) \subset \Omega$.
        \item If $\partial B_r(0) \cap \HP^+_0 \subseteq \partial \Omega \cap \HP^+_0$ then we have $\partial B_r(0) \cap \HP^+_0 \subseteq R^+_0(\partial \Omega)$ because of the earlier results. But then we can't have a sequence $\Tilde{z}_k \in \partial \Omega \cap \HP^+_0$ such that $\Tilde{z}_k \to 0 \in \partial \Omega$.
    \end{itemize}
\end{proof}

{
\begin{lem}\label{lem: J4bNontran}
    Suppose $\HP_0$ is a nontransversal plane and assumption (J4b) holds. {In addition to that, $\Omega$ satisfies assumption \ref{assum_domain}}. Then,
    $$|B_r \cap (\Omega \setminus R_0(\overline{\Omega}))|=0,$$ and there exists an interior touching point on $\partial \Omega \cap R^+_0(\partial \Omega)$
\end{lem}
}

\begin{proof}
Following the same steps as in the case of the constant kernel, we get 
    \begin{align*}
        0 &= H^J(\Tilde{z}_k) - H^J(z_k)= H^J(R_{\lambda^*_k}(z_k)) - H^J(z_k)\\
        &=2 \int_{\Omega \setminus R_{\lambda^*_k}(\Omega)}J(|y-z_k|) dy - 2 \int_{ R_{\lambda^*_k}(\Omega)  \setminus \Omega}J(|y-z_k|) dy \\
        &= 2 \int_{\Omega \setminus R_{\lambda^*_k}(\Omega)} \big( J(|y-z_k|) - J(|y-\tilde{z}_k|) \big) dy  \, \\
        &= 2 \int_{\Omega \setminus R_{\lambda^*_k}(\Omega)} \nabla J(|y-\xi_k|) \cdot (z_k - \tilde{z}_k) dy \quad [\text{Using mean value theorem}] \\
        &= 2 \int_{\Omega \setminus R_{\lambda^*_k}(\Omega)} \mu'(|y-\xi_k|) \frac{y-\xi_k}{|y-\xi_k|} \cdot (z_k - \tilde{z}_k) dy
    \end{align*}
    where $\xi_k = \tilde{z}_k + t(z_k - \tilde{z}_k)$ for some $0<t<1$. Then we have 
    \begin{align*}
        0 &= \frac{H^J(\Tilde{z}_k) - H^J(z_k)}{|z_k - \tilde{z}_k|} = 2 \int_{\Omega \setminus R_{\lambda^*_k}(\Omega)} \mu'(|y-\xi_k|) \frac{y-\xi_k}{|y-\xi_k|} \cdot \frac{(z_k - \tilde{z}_k)}{|z_k - \tilde{z}_k|} dy.
    \end{align*}
 Next, we observe that 
 $$(y-\xi_k) \cdot (z_k - \tilde{z}_k) = (y_n - \xi_{k,n})(z_{k,n}- \tilde{z}_{k,n}), \quad  |z_k - \tilde{z}_k| = |z_{k,n} - \tilde{z}_{k,n}|,  \quad  \frac{(z_{k,n}- \tilde{z}_{k,n})}{|z_{k,n} - \tilde{z}_{k,n}|} = -1.$$ Using the Dominated Convergence Theorem, we get 
    \begin{align*}
        \lim_{k \to 0} \frac{H^J(\Tilde{z}_k) - H^J(z_k)}{|z_k - \tilde{z}_k|} &=-2 \lim_{k \to 0} \int_{\Omega \setminus R_{\lambda^*_k}(\Omega)} \mu'(|y-\xi_k|) \frac{(y_n - \xi_{k,n})}{|y-\xi_k|}  dy \\
        &= 2 \int_{\Omega \setminus R_0(\Omega)} -\mu'(|y|) \frac{(y_n)}{|y|}  dy,
    \end{align*}    
    since $\mu' <0$ almost everywhere in $B_r$ and $y_n >0$ in $\Omega \setminus R_0(\overline{\Omega})$. Hence, we have 
    $$|B_r \cap (\Omega \setminus R_0(\overline{\Omega}))|=0.$$
    This is sufficient to conclude that we have an interior point.
    
\end{proof}

We collect the results above into a proof for Theorem~\ref{thm:main}.
\begin{proof}
    Assumption~\ref{assum_domain} allows us to apply Theorem~\ref{thm:monotone_implies_equality} to conclude that if $\ux,\ox\in\partial\Omega$ and $\ux\sim\ox$, then $H^J(\ux)=H^J(\ox)$. Suppose $\HP_{\lambda_0}$ is a critical hyperplane. Lemmas~\ref{lem: J4aNontran} and~\ref{lem: J4bNontran} imply there exists a tounching point $\ox\in \HP_{\lambda_0}^+\cap\partial\Omega$ such that $\ux=R_{\lambda_0}(\ox)\in\HP_{\lambda_0}^-\cap\partial\Omega$. The symmetry of $\Omega$ across $\HP_{\lambda_0}$ follows from Lemma~\ref{lem: interior}.
\end{proof}


\end{document}